\documentclass[12pt]{amsart}

\usepackage{framed}
\usepackage{braket}
\usepackage{amsmath,amssymb,amsthm}
\usepackage{color}
\usepackage{bm}
\usepackage[all]{xy}
\usepackage{amscd}
\usepackage{graphicx}
\usepackage{ascmac}
\usepackage{BOONDOX-frak, mathrsfs}
\usepackage[top=30truemm,bottom=30truemm,left=25truemm,right=25truemm]{geometry}
\usepackage{mathtools}
\mathtoolsset{showonlyrefs=true}
\usepackage{tikz}
\usetikzlibrary{cd}

\makeatletter
\@addtoreset{equation}{section}

\makeatother

\newcommand{\Z}{\mathbb{Z}}
\newcommand{\N}{\mathbb{N}}
\newcommand{\Q}{\mathbb{Q}}

\newcommand{\F}{\mathbb{F}}

\newcommand{\fF}{\mathfrak{F}}
\newcommand{\fG}{\mathfrak{G}}

\newcommand{\fp}{\mathfrak{p}}

\newcommand{\cB}{\mathcal{B}}

\newcommand{\cF}{\mathcal{F}}
\newcommand{\cG}{\mathcal{G}}

\newcommand{\cI}{\mathcal{I}}

\newcommand{\cY}{\mathcal{Y}}

\newcommand{\ol}[1]{\overline{#1}}

\newcommand{\lrang}[1]{\langle #1 \rangle}

\DeclareMathOperator{\Gal}{Gal}
\DeclareMathOperator{\Coker}{Cok}

\DeclareMathOperator{\Ker}{Ker}

\DeclareMathOperator{\fin}{fin}

\DeclareMathOperator{\ab}{ab}

\DeclareMathOperator{\id}{id}
\DeclareMathOperator{\Cl}{Cl}
\DeclareMathOperator{\ord}{ord}
\DeclareMathOperator{\Hom}{Hom}
\DeclareMathOperator{\Ext}{Ext}

\DeclareMathOperator{\rank}{rank}

\DeclareMathOperator{\cha}{char}

\usepackage[OT2,T1]{fontenc}
\DeclareSymbolFont{cyrletters}{OT2}{wncyr}{m}{n}
\DeclareMathSymbol{\Sha}{\mathalpha}{cyrletters}{"58}

\let\oldenumerate\enumerate
\renewcommand{\enumerate}{
   \oldenumerate
   \setlength{\itemsep}{1pt}
   \setlength{\parskip}{0pt}
   \setlength{\parsep}{0pt}
}
\let\olditemize\itemize
\renewcommand{\itemize}{
   \olditemize
   \setlength{\itemsep}{1pt}
   \setlength{\parskip}{0pt}
   \setlength{\parsep}{0pt}
}

\theoremstyle{plain}
\newtheorem{thm}{Theorem}[section]
\newtheorem{lem}[thm]{Lemma}

\newtheorem{prop}[thm]{Proposition}
\newtheorem{cor}[thm]{Corollary}

\theoremstyle{definition}
\newtheorem{defn}[thm]{Definition}

\newtheorem{rem}[thm]{Remark}
\newtheorem{eg}[thm]{Example}

\newcommand{\hatZ}{\hat{\Z}}
\newcommand{\hatDelta}{\widehat{\Delta}}
\newcommand{\dcoset}[3]{#1 \backslash #2 / #3}
\DeclareMathOperator{\tors}{tors}
\DeclareMathOperator{\length}{length}
\DeclareMathOperator{\height}{ht}
\DeclareMathOperator{\modu}{-mod}
\newcommand{\modC}[1]{(#1\modu)}

\makeatletter
\@namedef{subjclassname@2020}{\textup{2020} Mathematics Subject Classification}
\makeatother

\title
[Ideal class groups]
{Construction of Galois modules describing ideal class groups and applications to Iwasawa theory}
\author[T.~Kataoka]{Takenori Kataoka}
\address{Department of Mathematics, Faculty of Science Division II, Tokyo University of Science.
1-3 Kagurazaka, Shinjuku-ku, Tokyo 162-8601, Japan}
\email{tkataoka@rs.tus.ac.jp}
\date{\today}

\begin{document}

\maketitle

\begin{abstract}
For a finite Galois extension of number fields, the ideal class groups of intermediate fields do not behave well under taking Galois coinvariants in general.
In this paper, we construct a Galois module whose Galois coinvariants explicitly describe the ideal class groups of the intermediate fields.
As an application, we demonstrate how this construction naturally yields some results in classical Iwasawa theory.
\end{abstract}

\section{Introduction}\label{sec:intro}

Let $L/K$ be a finite Galois extension of number fields with Galois group $G = \Gal(L/K)$.
The ideal class group $\Cl(L)$ of $L$ is naturally equipped with an action of $G$.
The norm map $\Cl(L) \to \Cl(K)$ between the ideal class groups then induces a map
\[
\Cl(L)_G \to \Cl(K),
\]
where $\Cl(L)_G$ denotes the module of $G$-coinvariants of $\Cl(L)$.
This map can be used to study the structure of $\Cl(K)$ via $\Cl(L)_G$.
However, the difficulty is that this map is neither injective nor surjective in general.

In this paper, we construct a module equipped with an action of $G$, whose coinvariants determine the ideal class groups of the intermediate fields of $L/K$.
When $L/K$ is unramified, such a construction is classically known, and it indeed plays a role in studying the ideal class groups.
The main contribution of this paper is to deal with the case where $L/K$ is ramified.

We hope that this construction will have various applications to the study of ideal class groups.
In the present paper, we derive two results in classical Iwasawa theory.
The first is Iwasawa's asymptotic class number formula for $\Z_p$-extensions; indeed, this application was the starting point of the present work.
The second is a duality relation between the Iwasawa module, defined as the inverse limit of the ideal class groups, and their direct limit along $\Z_p^d$-extensions.

\subsection{Unramified case}\label{ss:intro_unr}

In Section \ref{sec:unr}, we consider the case where $L/K$ is unramified.
The following theorem is essentially due to Gruenberg \cite{Gru76} (see also Ritter--Weiss \cite{RW96}, \cite{RW02}), but our exposition differs somewhat from it.
The proof is given in Subsection \ref{ss:unr_pf}.

\begin{thm}\label{thm:main_unr}
Let $K$ be a number field.
For each finite unramified Galois extension $L/K$, setting $G = \Gal(L/K)$, we can construct a left $\Z[G]$-module $B(L) = B(L/K)$ satisfying the following:
For any intermediate field $K'$ of $L/K$ with $G' = \Gal(L/K')$, we have an exact sequence
\[
0 \to \Cl(K') \to B(L)_{G'} \to \Z[G' \backslash G] \to \Z \to 0.
\]
\end{thm}

Here and henceforth, $\Z[G' \backslash G] \to \Z$ and similar maps always denote the augmentation map.

In particular, by letting $K' = L$ and $K' = K$ respectively, we have an exact sequence
\begin{equation}\label{eq:fund_unr}
0 \to \Cl(L) \to B(L) \to \Z[G] \to \Z \to 0
\end{equation}
and an isomorphism
\begin{equation}\label{eq:fund_unr2}
\Cl(K) \simeq B(L)_G.
\end{equation}

We have $\Cl(K') \simeq (B(L)_{G'})_{\fin}$, where $(-)_{\fin}$ denotes the $\Z$-torsion submodule.
This shows that the single $\Z[G]$-module $B(L)$ determines $\Cl(K')$ for all intermediate fields $K'$ of $L/K$.
Note that the module $B(L)$ depends not only on $L$ but also on $K$; indeed, by \eqref{eq:fund_unr}, we have $\rank_{\Z} B(L) = [L: K] - 1$.

In Subsection \ref{ss:unr_app}, we see how this module $B(L)$ applies to the study of the ideal class groups.
In Subsection \ref{ss:unr_inf}, we show a variant of Theorem \ref{thm:main_unr} for the case where $L/K$ is an infinite extension.

\subsection{Ramified case}\label{ss:intro_ram}

In Section \ref{sec:ram}, we study the case where $L/K$ is possibly ramified.
The following is the main theorem and its proof is given in Subsection \ref{ss:ram_pf}.

\begin{thm}\label{thm:main_ram}
Let $K$ be a number field and $S$ a nonempty finite set of (finite or infinite) places of $K$.
For each finite Galois extension $L/K$ that is unramified outside $S$, setting $G = \Gal(L/K)$, we can construct a left $\Z[G]$-module $B^S(L) = B^S(L/K)$ satisfying the following:
For any intermediate field $K'$ of $L/K$ with $G' = \Gal(L/K')$, we have an exact sequence
\[
0 \to \Cl(K') \to B^S(L)_{G'} \to \bigoplus_{v \in S} \Z[\dcoset{G'}{G}{I_v}] \to \Z \to 0,
\]
where $I_v \subset G$ denotes the inertia subgroup of (a place of $L$ lying above) $v$.
\end{thm}

In particular, by letting $K' = L$ and $K' = K$ respectively, we have exact sequences
\begin{equation}\label{eq:fund_ram}
0 \to \Cl(L) \to B^S(L) \to \bigoplus_{v \in S} \Z[G/I_v] \to \Z \to 0
\end{equation}
and
\begin{equation}\label{eq:fund_ram2}
0 \to \Cl(K) \to B^S(L)_{G} \to \bigoplus_{v \in S} \Z \to \Z \to 0.
\end{equation}

The single $\Z[G]$-module $B^S(L)$ again determines $\Cl(K')$ by $\Cl(K') \simeq (B^S(L)_{G'})_{\fin}$.
Note that the set $S$ is assumed to be nonempty, so Theorem \ref{thm:main_unr} is not literally covered by Theorem \ref{thm:main_ram}.
(However, we can formally recover Theorem \ref{thm:main_unr} from Theorem \ref{thm:main_ram} by setting $S$ to be an arbitrary singleton.)

In Subsection \ref{ss:ram_app}, we briefly discuss its application to the study of the ideal class groups.
Then in Subsection \ref{ss:ram_inf}, we obtain a variant for infinite extensions $L/K$.
Because it is the key to the applications in Sections \ref{sec:app_Iw} and \ref{sec:duality}, we now state the result in the Iwasawa-theoretic situation.

In what follows, we fix a prime number $p$ and a number field $K$.
Let $K_{\infty}/K$ be a $\Z_p^d$-extension with some $d \geq 1$ and set $\Gamma = \Gal(K_{\infty}/K)$.
For each place $v$ of $K$, let $I_v \subset D_v \subset \Gamma$ be the inertia and decomposition groups.
We define the Iwasawa algebra by $\Z_p[[\Gamma]] = \varprojlim_{\Gamma'} \Z_p[\Gamma/\Gamma']$, where $\Gamma'$ runs over the open subgroups of $\Gamma$.
We define the unramified Iwasawa module by
\[
X_{K_{\infty}} = \varprojlim_{K'} A_{K'},
\]
where $K'$ runs over the intermediate number fields of $K_{\infty}/K$, we set $A_{K'} = \Z_p \otimes_{\Z} \Cl(K')$, and the transition maps are (the $p$-parts of) the norm maps.
Then $X_{K_{\infty}}$ is naturally equipped with a profinite $\Z_p[[\Gamma]]$-module structure.
It is well-known that $X_{K_{\infty}}$ is finitely generated and torsion over $\Z_p[[\Gamma]]$.
This property will be reproved in this paper (see Theorem \ref{thm:intro_duality}), but we accept this for now.

\begin{thm}\label{thm:main_ram_Iw}
Let $K_{\infty}/K$ be a $\Z_p^d$-extension with $\Gamma = \Gal(K_{\infty}/K)$.
Let $S$ be the set of (necessarily $p$-adic) primes of $K$ that are ramified in $K_{\infty}/K$.
Then we can construct a finitely generated torsion $\Z_p[[\Gamma]]$-module $Y_{K_{\infty}}$ satisfying exact sequences
\begin{equation}\label{eq:Iw_inf}
0 \to X_{K_{\infty}} \to Y_{K_{\infty}} \to \bigoplus_{v \in S} \Z_p[[\Gamma/I_v]] \to \Z_p \to 0
\end{equation}
and
\begin{equation}\label{eq:Iw_fin}
0 \to A_{K'} \to (Y_{K_{\infty}})_{\Gamma'} \to \bigoplus_{v \in S} \Z_p[\Gamma/\Gamma' I_v] \to \Z_p \to 0
\end{equation}
for any intermediate number field $K'$ of $K_{\infty}/K$ with $\Gamma' = \Gal(K_{\infty}/K')$.
\end{thm}

The natural map $(X_{K_{\infty}})_{\Gamma'} \to A_{K'}$ is not necessarily injective or surjective; in fact, $(X_{K_{\infty}})_{\Gamma'}$ is often even infinite.
Our theorem gives a precise relation between the Iwasawa module $X_{K_{\infty}}$ and the ideal class group $A_{K'}$ via the module $Y_{K_{\infty}}$.

\subsection{Variants}\label{ss:int_var}

In Section \ref{sec:variants}, we generalize the theorems above to some variants of the ideal class group, such as the ray class groups.
As an example, we consider the split Iwasawa module
\[
X'_{K_{\infty}} = \varprojlim_{K'} A'_{K'},
\]
where $A'_{K'} = \Z_p \otimes_{\Z} \Cl'(K')$ and $\Cl'(K')$ denotes the quotient of $\Cl(K')$ by the subgroup generated by the classes of the prime ideals containing $p$.
Then we have the following analogue of Theorem \ref{thm:main_ram_Iw}.

\begin{thm}\label{thm:main_ram_Iw_sp}
Let $K_{\infty}/K$ be a $\Z_p^d$-extension with $\Gamma = \Gal(K_{\infty}/K)$.
Let $S$ be the set of $p$-adic primes of $K$ that are not totally split in $K_{\infty}/K$.
Then we can construct a finitely generated torsion $\Z_p[[\Gamma]]$-module $Y'_{K_{\infty}}$ satisfying exact sequences
\begin{equation}
0 \to X'_{K_{\infty}} \to Y'_{K_{\infty}} \to \bigoplus_{v \in S} \Z_p[[\Gamma/D_v]] \to \Z_p \to 0
\end{equation}
and
\begin{equation}
0 \to A'_{K'} \to (Y'_{K_{\infty}})_{\Gamma'} \to \bigoplus_{v \in S} \Z_p[\Gamma/\Gamma' D_v] \to \Z_p \to 0
\end{equation}
for any intermediate number field $K'$ of $K_{\infty}/K$ with $\Gamma' = \Gal(K_{\infty}/K')$.
\end{thm}

Using this theorem instead of Theorem \ref{thm:main_ram_Iw}, the results in Sections \ref{sec:app_Iw} and \ref{sec:duality} described below are also valid for $\Cl'(K')$ and $X'_{K_{\infty}}$ instead of $\Cl(K')$ and $X_{K_{\infty}}$.

\subsection{Iwasawa's asymptotic class number formula}\label{ss:intro_ICNF}

In Section \ref{sec:app_Iw}, we apply Theorem \ref{thm:main_ram_Iw} to recover Iwasawa's asymptotic class number formula \cite[Theorem 11]{Iwa59} (see also Washington \cite[Theorem 13.13]{Was97}).
The statement is the following.

\begin{thm}\label{thm:intro_ICNF}
Let $K_{\infty}/K$ be a $\Z_p$-extension and let $K_n$ denote its $n$-th layer for each $n \geq 0$.
Then there exist integers $\lambda \geq 0$, $\mu \geq 0$, and $\nu$ such that
\[
\ord_p(\# \Cl(K_n))
= \lambda n + \mu p^n + \nu
\]
holds for $n \gg 0$.
Here, $\ord_p(-)$ denotes the $p$-adic valuation normalized by $\ord_p(p) = 1$.
\end{thm}

Our proof requires Theorem \ref{thm:main_ram_Iw} only as the arithmetic input.
It remains to apply an algebraic argument.
We take this opportunity to develop a systematic approach to Iwasawa-type asymptotic formulas via homological algebra and commutative algebra.
This enables us to avoid the use of the structure theorem for $\Z_p[[T]]$-modules, and to naturally specify the range of $n$ for which the asymptotic formula holds (see Theorem \ref{thm:main_ICNF} for the precise statement).

Here we only study the case of $\Z_p$-extensions.
It seems to be plausible to apply this idea to $\Z_p^d$-extensions, but we do not study it in this paper.
Even more generally, the natural variant of Theorem \ref{thm:main_ram_Iw} for $p$-adic Lie extensions is valid.
We hope that this construction will be useful in extending Iwasawa's asymptotic formula to more general extensions.
Such a formula is obtained by Cuoco--Monsky \cite{CM81} for $\Z_p^d$-extensions, and Lei \cite{Lei17} and Liang--Lim \cite{LL19} study the case $\Z_p^d \rtimes \Z_p$.

\subsection{Duality of class groups along $\Z_p^d$-extensions}\label{ss:intro_duality}

In Section \ref{sec:duality}, we deduce a duality theorem in Iwasawa theory as an application of Theorem \ref{thm:main_ram_Iw}.
Let $K_{\infty}/K$ be a $\Z_p^d$-extension with $\Gamma = \Gal(K_{\infty}/K)$ and we set $\Lambda = \Z_p[[\Gamma]]$.
We define
\[
A_{K_{\infty}} = \varinjlim_{K'} A_{K'},
\]
where $K'$ runs over the intermediate number fields of $K_{\infty}/K$ and the transition maps are (the $p$-parts of) the capitulation maps.
Then $A_{K_{\infty}}$ is a discrete $\Lambda$-module and its Pontryagin dual $A_{K_{\infty}}^{\vee}$ is a profinite $\Lambda$-module.
Here, for a discrete or profinite $\Z_p$-module $M$ with an action of an abelian group $G$, we equip its Pontryagin dual $M^{\vee} = \Hom_{\Z_p}(M, \Q_p/\Z_p)$ with an action of $G$ by
\[
(\sigma f)(x) = f(\sigma x)
\]
for $f \in M^{\vee}$, $\sigma \in G$, and $x \in M$.

The main theorem is the following.

\begin{thm}\label{thm:intro_duality}
Let $K_{\infty}/K$ be a $\Z_p^d$-extension.
Then $A_{K_{\infty}}^{\vee}$ and $X_{K_{\infty}}$ are both finitely generated torsion over $\Lambda$ and pseudo-isomorphic to each other.
\end{thm}

This theorem is not new, and indeed it is one of the main theorems of Lai--Tan \cite[Theorem 2]{LT21}.
However, the proof there requires a delicate argument, while the proof in the present paper, outlined below, appears to be more streamlined.
As explained in the introduction of \cite{LT21}, Theorem \ref{thm:intro_duality} for $d = 1$ is proved by Iwasawa \cite[page 266, Theorem 11]{Iwa73}, and partial results for general $d \geq 1$ had been obtained by Nekov\'{a}\v{r} \cite[\S 9.4, \S 9.5]{Nek06} and Vauclair \cite[Theorem 6.7 and Corollary 6.8]{Vau09}.

Our proof proceeds as follows.
The arithmetic input is again Theorem \ref{thm:main_ram_Iw} only.
Because of \eqref{eq:Iw_fin}, we have $A_{K'} \simeq ((Y_{K_{\infty}})_{\Gamma'})_{\fin}$ and so 
\[
X_{K_{\infty}} \simeq \varprojlim_{\Gamma'} ((Y_{K_{\infty}})_{\Gamma'})_{\fin}
\]
and
\[
A_{K_{\infty}}^{\vee} 
\simeq \left( \varinjlim_{\Gamma'} ((Y_{K_{\infty}})_{\Gamma'})_{\fin} \right)^{\vee}
\simeq \varprojlim_{\Gamma'} (((Y_{K_{\infty}})_{\Gamma'})_{\fin})^{\vee},
\]
where $\Gamma'$ runs over the open subgroups of $\Gamma$.
These lead us to study two functors $\fG$ and $\fF$ defined by
\[
\fG(M) = \varprojlim_{\Gamma'} (M_{\Gamma'})_{\fin},
\quad
\fF(M) = \varprojlim_{\Gamma'} ((M_{\Gamma'})_{\fin})^{\vee}.
\]
We will indeed prove that, for any finitely generated $\Lambda$-module $M$, both $\fG(M)$ and $\fF(M)$ are finitely generated torsion over $\Lambda$ and pseudo-isomorphic to each other.
This claim sounds non-trivial since $\fF$ is contravariant while $\fG$ is covariant.
When $d = 1$, Lee \cite{Lee20} studied the functors $\fG$ and $\fF$ in a somewhat different context, and proved this claim \cite[Corollary A.2.13]{Lee20}.
For $d \geq 1$, Lim \cite[Corollary 2.5]{Lim24} showed this claim, assuming $M_{\Gamma'}$ are all finite; this assumption does not hold in our case, but does hold in the higher $K$-theoretic version, which is the theme in \cite{Lim24}.

We will also observe that $A_{K_{\infty}}^{\vee}$ has no nonzero pseudo-null submodules.
We then deduce the following, which is again not new (see Lai--Tan \cite[Corollary 1.2.1]{LT21}).

\begin{cor}\label{cor:GGC_equiv}
The following are equivalent.
\begin{itemize}
\item[(i)]
The $\Lambda$-module $X_{K_{\infty}}$ is pseudo-null.
\item[(ii)]
The $\Lambda$-module $A_{K_{\infty}}^{\vee}$ is pseudo-null.
\item[(iii)]
The $\Lambda$-module $A_{K_{\infty}}$ is trivial.
\end{itemize}
\end{cor}

There is a classical conjecture of Greenberg \cite[Conjecture 3.5]{Gree01} concerning the pseudo-nullity of the Iwasawa modules for $\Z_p^d$-extensions.
This corollary gives an equivalent condition for the conjecture.
Moreover, Theorem \ref{thm:main_ram_Iw} itself provides a reformulation of the pseudo-nullity in terms of $Y_{K_{\infty}}$, which might be easier to study compared to $X_{K_{\infty}}$ itself since $Y_{K_{\infty}}$ behaves well under Galois coinvariants.

As noted in Subsection \ref{ss:int_var}, both Theorem \ref{thm:intro_duality} and Corollary \ref{cor:GGC_equiv} are valid for $X'_{K_{\infty}}$ instead of $X_{K_{\infty}}$ as well.
This recovers other results of \cite{LT21}, \cite{Nek06}, and \cite{Vau09}.

\section{Unramified case}\label{sec:unr}

In this section, we review the work of Gruenberg from our perspective.
We prove Theorem \ref{thm:main_unr} in Subsection \ref{ss:unr_pf}, and then discuss some applications in Subsection \ref{ss:unr_app}.
In Subsection \ref{ss:unr_inf}, we show that the same method applies to infinite extensions.
This section contains no essentially new results, and serves as a prototype for the next section.

\subsection{Proof of Theorem \ref{thm:main_unr}}\label{ss:unr_pf}

The theorem essentially follows by applying Gruenberg's ``translation functor'' \cite[Section 10.5]{Gru76} (see also \cite[Proposition 1]{RW96}) from group extensions to module extensions.

Let $K$ be a number field.
For each finite extension $L/K$, let $H(L)$ be the Hilbert class field of $L$, that is, the maximal unramified abelian extension.
Then by class field theory, the reciprocity map gives an isomorphism $\Cl(L) \simeq \Gal(H(L)/L)$.

We now construct the module $B(L)$.

\begin{defn}
Let $L/K$ be a finite unramified Galois extension and set $G = \Gal(L/K)$.
We take an auxiliary finite extension $F/L$ satisfying the following conditions:
\begin{itemize}
\item[(a)]
$F\supset H(L)$,
\item[(b)]
$F/L$ is an unramified extension, and
\item[(c)]
$F/K$ is a Galois extension.
\end{itemize}
For example, the minimal choice is $F = H(L)$.
We then set $\cG = \Gal(F/K)$ (by condition (c)).
Let $\Delta \cG$ denote the augmentation ideal of $\Z[\cG]$, which we regard as a left $\Z[\cG]$-module.
We define
\[
B(L) = (\Delta \cG)_{\Gal(F/L)}
= \Z[G] \otimes_{\Z[\cG]} \Delta \cG,
\]
which is a left $\Z[G]$-module.
This $B(L)$ is independent of the choice of $F$, up to isomorphism (see Remark \ref{rem:B_unr_indep}).
\end{defn}

We now show that this module $B(L)$ satisfies the property stated in Theorem \ref{thm:main_unr}.

\begin{proof}[Proof of Theorem \ref{thm:main_unr}]
As in the statement, let $K'$ be any intermediate field of $L/K$ and set $G' = \Gal(L/K')$.
We also set $\cG' = \Gal(F/K')$.
The situation is described as follows.
\begin{center}
\begin{tikzpicture}
  \node (L)  at (0,3) {$L$};
  \node (Kp) at (0,1.5) {$K'$};
  \node (K)  at (0,0) {$K$};
  \node (F)  at (2,4) {$F$};

  \draw (F)  --  (L);
  \draw (K)  --  (Kp);
  \draw[bend left=50] (K)  to node[left]  {$G$} (L);
  \draw (Kp) -- node[left]  {$G'$} (L);

  \draw[bend right=20] (K)  to node[right] {$\cG$} (F);
  \draw (Kp) -- node[below] {$\cG'$} (F);
\end{tikzpicture}
\end{center}
The exact sequence
\[
0 \to \Delta \cG \to \Z[\cG] \to \Z \to 0
\]
induces, by taking $\Z \otimes_{\Z[\cG']} (-)$, an exact sequence
\[
0 \to H_1(\cG', \Z) \to (\Delta \cG)_{\cG'} \to \Z[\cG' \backslash \cG] \to \Z \to 0.
\]
Here, we used the fact that $H_1(\cG', \Z[\cG]) = 0$ since $\Z[\cG]$ is free as a left $\Z[\cG']$-module.
We now observe $\cG' \backslash \cG \simeq G' \backslash G$ (as sets) and
\[
(\Delta \cG)_{\cG'} 
\simeq ((\Delta \cG)_{\Gal(F/L)})_{G'}
= B(L)_{G'}.
\]
We also have
\[
H_1(\cG', \Z) 
\simeq \cG'_{\ab}
\simeq \Gal(H(K')/K')
\simeq \Cl(K'),
\]
where the middle isomorphism holds since $F/K'$ is unramified (by condition (b) and $L/K$ being unramified) and $F \supset H(L) \supset H(K')$ (by condition (a)).
This completes the proof of Theorem \ref{thm:main_unr}.
\end{proof}

\begin{rem}\label{rem:B_unr_indep}
Let us show that the $\Z[G]$-module $B(L)$ is independent of the choice of $F$, up to isomorphism.
Let $F_0 = H(L)$ be the minimal choice, so $F_0 \subset F$ holds.
Set $\cG_0 = \Gal(F_0/K)$.
Then we have a commutative diagram
\[
\xymatrix{
0 \ar[r] & \Delta \cG \ar[r] \ar@{->>}[d] & \Z[\cG] \ar[r] \ar@{->>}[d] & \Z \ar[r] \ar@{=}[d] & 0\\
0 \ar[r] & \Delta \cG_0 \ar[r] & \Z[\cG_0] \ar[r] & \Z \ar[r] & 0,
}
\]
which then induces a commutative diagram
\[
\xymatrix{
0 \ar[r] & H_1(\Gal(F/L), \Z) \ar[r] \ar[d]_{\simeq} & (\Delta \cG)_{\Gal(F/L)} \ar[r] \ar@{->>}[d] & \Z[G] \ar[r] \ar@{=}[d] & \Z \ar[r] \ar@{=}[d] & 0\\
0 \ar[r] & H_1(\Gal(F_0/L), \Z) \ar[r] & (\Delta \cG_0)_{\Gal(F_0/L)} \ar[r] & \Z[G] \ar[r] & \Z \ar[r] & 0.
}
\]
Here, the left vertical map is an isomorphism since both are identified with $\Gal(H(L)/L) \simeq \Cl(L)$ as in the proof above (these two horizontal sequences are nothing but \eqref{eq:fund_unr}).
Therefore, by the five lemma, the surjective homomorphism $(\Delta \cG)_{\Gal(F/L)} \to (\Delta \cG_0)_{\Gal(F_0/L)}$ is indeed an isomorphism.
\end{rem}

By construction, we obtain the following functorial properties.

\begin{lem}\label{lem:funct}
Let $K'' \supset K'$ be intermediate fields of $L/K$ and set $G'' = \Gal(L/K'') \subset G' = \Gal(L/K')$.
The following hold.
\begin{itemize}
\item[(1)]
We have a commutative diagram
\[
\xymatrix{
0 \ar[r] & \Cl(K'') \ar[r] \ar[d] & B(L)_{G''} \ar[r] \ar@{->>}[d] & \Z[G'' \backslash G] \ar[r] \ar@{->>}[d] & \Z \ar[r] \ar@{=}[d] & 0\\
0 \ar[r] & \Cl(K') \ar[r] & B(L)_{G'} \ar[r] & \Z[G' \backslash G] \ar[r] & \Z \ar[r] & 0,
}
\]
where the left vertical arrow is the norm map and the other vertical arrows are the natural ones.
\item[(2)]
We have a commutative diagram
\[
\xymatrix{
0 \ar[r] & \Cl(K'') \ar[r] & B(L)_{G''} \ar[r] & \Z[G'' \backslash G] \ar[r] & \Z \ar[r] & 0\\
0 \ar[r] & \Cl(K') \ar[r] \ar[u]& B(L)_{G'} \ar[r] \ar[u] & \Z[G' \backslash G] \ar[r] \ar@{^(->}[u] & \Z \ar[r] \ar@{^(->}[u] & 0,
}
\]
where the left vertical arrow is the capitulation map and the other vertical arrows are induced by a norm element of $G'/G''$; in particular, the right-most map is the multiplication by $[G': G'']$.
\end{itemize}
\end{lem}

\begin{proof}
It is known that the reciprocity isomorphism
\[
\Cl(K') \simeq \Gal(H(K')/K') \simeq H_1(\Gal(F/K'), \Z)
\]
satisfies the commutative diagrams
\[
\xymatrix{
	\Cl(K'') \ar[r]^-{\simeq} \ar[d] & H_1(\Gal(F/K''), \Z) \ar[d]\\
	\Cl(K') \ar[r]_-{\simeq} & H_1(\Gal(F/K'), \Z),
}
\quad
\xymatrix{
	\Cl(K'') \ar[r]^-{\simeq} & H_1(\Gal(F/K''), \Z)\\
	\Cl(K') \ar[r]_-{\simeq} \ar[u] & H_1(\Gal(F/K'), \Z). \ar[u]
}
\]
Here, the left vertical maps are the norm and capitulation maps, respectively.
The right vertical maps are the homomorphism induced by the inclusion $\Gal(F/K'') \hookrightarrow \Gal(F/K')$ and the transfer map, respectively.
Claims (1) and (2) follow from these diagrams, respectively.
\end{proof}

\subsection{Consequences}\label{ss:unr_app}

We briefly discuss how Theorem \ref{thm:main_unr} applies to the study of ideal class groups.
Throughout this subsection, we suppose that $L/K$ is a finite unramified Galois extension and set $G = \Gal(L/K)$.

\subsubsection{The norm map}

The kernel and the cokernel of the norm map $\Cl(L)_G \to \Cl(K)$ are described as follows.

\begin{prop}\label{prop:norm_unr}
We have an exact sequence
\[
H_2(G, \Z) \to \Cl(L)_G \to \Cl(K) \to H_1(G, \Z) \to 0.
\]
\end{prop}

\begin{proof}
This follows upon applying $H_{\bullet}(G, -)$ to \eqref{eq:fund_unr}, combined with \eqref{eq:fund_unr2} and Lemma \ref{lem:funct}(1).
\end{proof}

\begin{eg}
Suppose moreover that $G$ is a cyclic group.
Then we have $H_2(G, \Z) = 0$ and $H_1(G, \Z) \simeq G_{\ab} \simeq G$, so we conclude an exact sequence
\[
0 \to \Cl(L)_G \to \Cl(K) \to G \to 0.
\]
We also obtain
\[
\# \Cl(L)^G 
= \# \Cl(L)_G 
= \# \Cl(K) \cdot \frac{1}{\# G},
\]
where the first equality holds since $G$ is cyclic.
This is a very special case of Chevalley's ambiguous class number formula.
\end{eg}

\subsubsection{The capitulation map}

The capitulation map $\Cl(K) \to \Cl(L)$ induces a map
\[
\Cl(K) \to \Cl(L)^G,
\]
where $\Cl(L)^G$ denotes the $G$-invariant part of $\Cl(L)$.
Its kernel and cokernel are described via the Tate cohomology groups of $B(L)$ as follows.

\begin{prop}\label{prop:unr_cap}
We have an exact sequence
\[
0 \to \hat{H}^{-1}(G, B(L)) \to \Cl(K) \to \Cl(L)^G \to \hat{H}^0(G, B(L)) \to 0.
\]
\end{prop}

\begin{proof}
This follows from the commutative diagram
\[
\xymatrix{
\Cl(L)^G \ar[r]^{\simeq} & B(L)^G \\
\Cl(K) \ar[r]_{\simeq} \ar[u] & B(L)_G \ar[u]
}
\]
obtained by Lemma \ref{lem:funct}(2).
The bottom isomorphism is \eqref{eq:fund_unr2}.
The top isomorphism is induced by \eqref{eq:fund_unr} since the induced map $\Z[G]^G \to \Z^G$ is injective.
The right vertical arrow is induced by the norm element of $G$, so its kernel and cokernel are the Tate cohomology groups as claimed.
\end{proof}

\begin{rem}\label{rem:Suzuki}
This description of the capitulation kernel as $\hat{H}^{-1}(G, B(L))$ was used by Suzuki \cite{Suz91} to show the following theorem:
\begin{quote}
For an unramified abelian extension $L/K$, we have
\[
[L: K] \mid \# \Ker(\Cl(K) \to \Cl(L)).
\]
\end{quote}
This is a simultaneous generalization of Hilbert 94 (the case where $L/K$ is cyclic) and the principal ideal theorem (the case where $L = H(K)$).
Gruenberg--Weiss \cite{GW00} reproved and refined this result.
\end{rem}

\subsection{Analogues for infinite extensions}\label{ss:unr_inf}

It is straightforward to generalize Theorem \ref{thm:main_unr} to the case where $L/K$ is an infinite extension.
Note that the existence of an infinite unramified extension $L/K$ is a subtle problem, called the class field tower problem, which is resolved by Golod--Shafarevich \cite{GS64}.

Let $\hatZ$ be the profinite completion of the ring $\Z$.
For a profinite group $G$, we define the completed group ring by
\[
\hatZ[[G]] = \varprojlim_U \hatZ[G/U],
\]
where $U$ runs over the open normal subgroups of $G$.
When $G'$ is a closed subgroup of $G$, we set
\[
\hatZ[[G' \backslash G]] 
= \hatZ \otimes_{\hatZ[[G']]} \hatZ[[G]].
\]
Similarly, when $I$ is another closed subgroup of $G$, we set
\[
\hatZ[[G' \backslash G/I]] 
= \hatZ \otimes_{\hatZ[[G']]} \hatZ[[G]] \otimes_{\hatZ[[I]]} \hatZ.
\]

Let $K$ be a number field.
For a (possibly infinite) algebraic extension $L/K$, we define
\[
X(L) = \varprojlim_{K'} \Cl(K')
\]
with respect to the norm maps, where $K'$ runs over the intermediate number fields of $L/K$.
Then $X(L)$ is a profinite $\hatZ$-module.
We also write $H(L)$ for the maximal unramified abelian extension of $L$.
We have an isomorphism $X(L) \simeq \Gal(H(L)/L)$.

We now construct a module $Y(L)$, which is a profinite variant of $B(L)$.

\begin{defn}\label{defn:Y(L)}
Let $L/K$ be a (possibly infinite) unramified Galois extension and set $G = \Gal(L/K)$.
We take an auxiliary (possibly infinite) extension $F/L$ satisfying the following conditions:
\begin{itemize}
\item[(a)]
$F\supset H(L)$,
\item[(b)]
$F/L$ is an unramified extension, and
\item[(c)]
$F/K$ is a Galois extension.
\end{itemize}
For example, the minimal choice is $F = H(L)$.
Alternatively, the maximal choice is $F = K_{\emptyset}$, the maximal unramified (not necessarily abelian) extension of $K$.
We then set $\cG = \Gal(F/K)$ (by condition (c)).
Let $\hatDelta \cG$ denote the augmentation ideal of $\hatZ[[\cG]]$.
We now define
\[
Y(L) = (\hatDelta \cG)_{\Gal(F/L)}
= \hatZ[[G]] \otimes_{\hatZ[[\cG]]} \hatDelta \cG,
\]
which is a left profinite $\hatZ[[G]]$-module.
This $Y(L)$ is independent of the choice of $F$, up to isomorphism (this follows in the same way as Remark \ref{rem:B_unr_indep}).
\end{defn}

If $L/K$ is a finite extension, we have $Y(L) = \hatZ \otimes_{\Z} B(L)$ by construction.
This module $Y(L)$ satisfies the following.

\begin{thm}\label{thm:main_unr_inf}
For any intermediate field $K'$ of $L/K$ with $G' = \Gal(L/K')$, we have an exact sequence
\[
0 \to X(K') \to Y(L)_{G'} \to \hatZ[[G' \backslash G]] \to \hatZ \to 0.
\]
\end{thm}

\begin{proof}
The same argument as in the proof of Theorem \ref{thm:main_unr} applies.
The exact sequence
\[
0 \to \hatDelta \cG \to \hatZ[[\cG]] \to \hatZ \to 0
\]
induces an exact sequence
\[
0 \to H_1(\cG', \hatZ) \to (\hatDelta \cG)_{\cG'} \to \hatZ[[\cG' \backslash \cG]] \to \hatZ \to 0
\]
with $\cG' = \Gal(F/K')$.
This is identified with the desired exact sequence.
\end{proof}

\section{Ramified case}\label{sec:ram}

In this section, we establish the ramified version of the results in Section \ref{sec:unr}.
The organization of this section closely follows that of Section \ref{sec:unr}.

\subsection{Proof of Theorem \ref{thm:main_ram}}\label{ss:ram_pf}

The following construction is the heart of this paper.

\begin{defn}\label{defn:B_S}
Let $K$ be a number field and $S$ a nonempty finite set of places of $K$.
Let $L/K$ be a finite Galois extension that is unramified outside $S$ and set $G = \Gal(L/K)$.
We take an auxiliary finite extension $F/L$ satisfying the following conditions:
\begin{itemize}
\item[(a)]
$F \supset H(L)$,
\item[(b)]
$F/L$ is unramified outside $S$, and
\item[(c)]
$F/K$ is a Galois extension.
\end{itemize}
For example, the minimal choice is $F = H(L)$.
We then set $\cG = \Gal(F/K)$ (by condition (c)).

For each $v \in S$, we take a place $u_v$ of $F$ lying above $v$, and write $w_v$ for the place of $L$ lying below $u_v$.
We then write $\cI_v \subset \cG$ and $I_v \subset G$ for the inertia subgroup of $u_v$ and $w_v$, respectively.
Note that the natural map $\cG \twoheadrightarrow G$ induces a surjective map $\cI_v \twoheadrightarrow I_v$.

We now define a left $\Z[\cG]$-module $\cB^{S, F}(L)$ as the kernel of the augmentation map
\[
\bigoplus_{v \in S} \Z[\cG/\cI_v] \to \Z,
\]
which is surjective since $S \neq \emptyset$.
We finally define
\[
B^S(L) = (\cB^{S, F}(L))_{\Gal(F/L)}
= \Z[G] \otimes_{\Z[\cG]} \cB^{S, F}(L),
\]
which is a left $\Z[G]$-module.
This $B^S(L)$ is independent of the choice of $F$, up to isomorphism (this follows in the same way as Remark \ref{rem:B_unr_indep}).
\end{defn}

It remains to show the exact sequence appearing in Theorem \ref{thm:main_ram}.
The two extreme cases \eqref{eq:fund_ram} and \eqref{eq:fund_ram2} are indeed easier to prove.
However, in order to show the general case, we have to deal with the double cosets in $\dcoset{\cG'}{\cG}{\cI_v}$.
We begin with two lemmas.

\begin{lem}\label{lem:dcoset}
Let $\cG$ be a finite group.
Let $\cG'$ and $\cI$ be subgroups of $\cG$.
Let $\sigma_1, \dots, \sigma_s \in \cG$ be a complete set of representatives of $\dcoset{\cG'}{\cG}{\cI}$.
Set $\cI^{\sigma} := \sigma \cI \sigma^{-1}$ for each $\sigma \in \cG$.
Then we have an isomorphism of left $\Z[\cG']$-modules
\[
\bigoplus_{j=1}^s \Z[\cG'/(\cI^{\sigma_j} \cap \cG')] \simeq \Z[\cG/\cI]
\]
induced by $\times \sigma_j$ at the $j$-th component.
\end{lem}

\begin{proof}
This is equivalent to the claim that the map
\[
\coprod_{j=1}^s \cG'/ (\cI^{\sigma_j} \cap \cG') \to \cG/\cI,
\]
induced by $\times \sigma_j$ at the $j$-th component, is a well-defined bijection.

We take any $j$ and any elements $\sigma', \sigma'' \in \cG'$.
Then
\begin{align*}
\sigma' \sigma_j \equiv \sigma'' \sigma_j \text{ in $\cG/\cI$}
& \Leftrightarrow
\sigma_j^{-1} (\sigma'')^{-1} \sigma' \sigma_j \in \cI \\
& \Leftrightarrow
(\sigma'')^{-1} \sigma' \in \cI^{\sigma_j} \cap \cG'\\
& \Leftrightarrow
\sigma' \equiv \sigma'' \text{ in $\cG'/(\cI^{\sigma_j} \cap \cG')$}.
\end{align*}
This shows the well-definedness as well as the injectivity on the $j$-th component.

For any element $\sigma \in \cG$, by the choice of $\sigma_1, \dots, \sigma_s$, there exist $j \in \{1, \dots, s\}$, $\sigma' \in \cG'$, and $\tau \in \cI$ such that $\sigma = \sigma' \sigma_j \tau$.
This immediately shows the surjectivity.
Since $j$ is unique, the injectivity also follows.
\end{proof}

\begin{lem}\label{lem:dcoset2}
In the same situation as in Lemma \ref{lem:dcoset}, we can construct the left vertical isomorphism in the following commutative diagram
\[
\xymatrix{
H_1(\cG', \Z[\cG/\cI]) \ar[r] \ar[d]_{\simeq} & H_1(\cG', \Z) \ar[d]^{\simeq} \\
\bigoplus_{j=1}^s (\cI^{\sigma_j} \cap \cG')_{\ab} \ar[r] & \cG'_{\ab},
}
\]
where the upper map is induced by the augmentation map $\Z[\cG/\cI] \to \Z$, the lower map by the inclusion map $\cI^{\sigma_j} \cap \cG' \to \cG'$, and the right vertical arrow is the standard one.
\end{lem}

\begin{proof}
We first note that the homomorphism
\[
\bigoplus_{j=1}^s \Z[\cG'/(\cI^{\sigma_j} \cap \cG')] 
\simeq \Z[\cG/\cI]
\to \Z,
\]
where the first isomorphism is constructed in Lemma \ref{lem:dcoset}, coincides with the augmentation map.
Then we obtain a commutative diagram
\[
\xymatrix{
\bigoplus_{j=1}^s H_1(\cG', \Z[\cG'/(\cI^{\sigma_j} \cap \cG')]) \ar[r]^-{\simeq} \ar[d]_{\simeq}
& H_1(\cG', \Z[\cG/\cI]) \ar[r] & H_1(\cG', \Z) \ar[dd]^{\simeq} \\
\bigoplus_{j=1}^s H_1(\cI^{\sigma_j} \cap \cG', \Z) \ar[d]_{\simeq} & &\\
\bigoplus_{j=1}^s (\cI^{\sigma_j} \cap \cG')_{\ab} \ar[rr] && \cG'_{\ab},
}
\]
where the vertical arrows are the standard ones.
This shows the lemma.
\end{proof}

\begin{eg}
Let $\cG = S_3$ be the symmetric group on the set $\{1, 2, 3\}$.
Let $\cG' = \cI = \lrang{(1 \, 2)}$.
In this case, we have $s = 2$ and may take $\sigma_1 = \id, \sigma_2 = (1 \, 3) \in S_3$ as a complete set of representatives of $\dcoset{\lrang{(1 \, 2)}}{S_3}{\lrang{(1 \, 2)}}$.
Then $\cI^{\sigma_1} \cap \cG' = \lrang{(1 \, 2)}, \cI^{\sigma_2} \cap \cG' = \{ \id\}$ and the consequence of Lemma \ref{lem:dcoset2} is a commutative diagram
\[
\xymatrix{
H_1(\cG', \Z[\cG/\cI]) \ar[r] \ar[d]_{\simeq} & H_1(\cG', \Z) \ar[d]^{\simeq} \\
\lrang{(1 \, 2)} \oplus \{\id\} \ar[r] & \lrang{(1 \, 2)}.
}
\]
\end{eg}

We are now ready to prove the main theorem.

\begin{proof}[Proof of Theorem \ref{thm:main_ram}]
Let $G' = \Gal(L/K')$ be as in the theorem and set $\cG' = \Gal(F/K')$.
The exact sequence
\[
0 \to \cB^{S, F}(L) \to \bigoplus_{v \in S} \Z[\cG/\cI_v] \to \Z \to 0
\]
obtained by the definition of $\cB^{S, F}(L)$ induces an exact sequence
\[
\bigoplus_{v \in S} H_1(\cG', \Z[\cG/\cI_v]) \to H_1(\cG', \Z)
\to B^S(L)_{G'} \to \bigoplus_{v \in S} \Z[\dcoset{G'}{G}{I_v}] \to \Z \to 0.
\]
It remains to show that the cokernel of the first homomorphism between the first homology groups is isomorphic to $\Cl(K')$.
For each $v \in S$, we take a complete set of representatives $\sigma_{v, 1}, \dots, \sigma_{v, s_v} \in \cG$ of $\dcoset{\cG'}{\cG}{\cI_v}$.
Then Lemma \ref{lem:dcoset2} gives us a commutative diagram
\[
\xymatrix{
\bigoplus_{v \in S} H_1(\cG', \Z[\cG/\cI_v]) \ar[r] \ar[d]_{\simeq} & H_1(\cG', \Z) \ar[d]^{\simeq} \\
\bigoplus_{v \in S} \bigoplus_{j=1}^{s_v} (\cI_v^{\sigma_{v, j}} \cap \cG')_{\ab} \ar[r] & \cG'_{\ab}.
}
\]
Here, the left vertical arrow depends on the choice of $\sigma_{v, j}$, while the other maps are the standard ones.

We notice that $\cI_v^{\sigma_{v, j}} \cap \cG' = \sigma_{v, j} \cI_v \sigma_{v, j}^{-1} \cap \cG'$ coincides with the inertia group in $\cG'$ of the place $\sigma_{v, j}(u_v)$ of $F$.
Since $\sigma_{v, 1}, \dots, \sigma_{v, s_v}$ is a set of representatives of $\dcoset{\cG'}{\cG}{\cI_v}$, for any place $v'$ of $K'$ lying above $v$, there exists $j \in \{1, \dots, s_v\}$ such that $v'$ lies below $\sigma_{v, j}(u_v)$.
It follows that the images of $(\cI_v^{\sigma_{v, j}} \cap \cG')_{\ab} \to \cG'_{\ab}$ for $j \in \{1, \dots, s_v\}$ are precisely the inertia groups of the places of $K'$ lying above $v$, possibly with repetitions.

Consequently, the cokernel of the lower map in the diagram coincides with
\[
\Gal(H(K')/K') \simeq \Cl(K')
\]
by conditions (a) and (b).
This completes the proof of Theorem \ref{thm:main_ram}.
\end{proof}

Let us record the dependency on the set $S$.

\begin{lem}\label{lem:dep_S}
In Definition \ref{defn:B_S}, let $S'$ be a finite set of places of $K$ such that $S' \supset S$.
Then we have a split exact sequence
\[
0 \to B^S(L) \to B^{S'}(L) \to \bigoplus_{v \in S' \setminus S} \Z[G] \to 0.
\]
\end{lem}

\begin{proof}
We may use the same auxiliary field $F$ for the construction of $B^{S'}(L)$ as that for $B^S(L)$.
By their definitions, we have a commutative diagram
\[
\xymatrix{
0 \ar[r] & \cB^{S, F}(L) \ar[r] \ar@{^(->}[d] & \bigoplus_{v \in S} \Z[\cG/\cI_v] \ar[r] \ar@{^(->}[d] & \Z \ar[r] \ar@{=}[d]& 0\\
0 \ar[r] & \cB^{S', F}(L) \ar[r] & \bigoplus_{v \in S'} \Z[\cG/\cI_v] \ar[r] & \Z \ar[r] & 0,
}
\]
which induces an exact sequence
\[
0 \to \cB^{S, F}(L) \to \cB^{S', F}(L) \to \bigoplus_{v \in S' \setminus S} \Z[\cG] \to 0
\]
since $\cI_v$ is trivial for $v \in S' \setminus S$.
This sequence of left $\Z[\cG]$-modules splits, so we obtain the lemma by taking $(-)_{\Gal(F/L)}$.
\end{proof}

The analogue of Lemma \ref{lem:funct} is valid as well.
For completeness, we record the statements.

\begin{lem}\label{lem:funct_ram}
Let $K'' \supset K'$ be intermediate fields of $L/K$ and set $G'' = \Gal(L/K'') \subset G' = \Gal(L/K')$.
The following hold.
\begin{itemize}
\item[(1)]
We have a commutative diagram
\[
\xymatrix{
0 \ar[r] & \Cl(K'') \ar[r] \ar[d] & B^S(L)_{G''} \ar[r] \ar@{->>}[d] & \bigoplus_{v \in S} \Z[\dcoset{G''}{G}{I_v}] \ar[r] \ar@{->>}[d] & \Z \ar[r] \ar@{=}[d] & 0\\
0 \ar[r] & \Cl(K') \ar[r] & B^S(L)_{G'} \ar[r] & \bigoplus_{v \in S} \Z[\dcoset{G'}{G}{I_v}] \ar[r] & \Z \ar[r] & 0,
}
\]
where the left vertical arrow is the norm map and the other vertical arrows are the natural ones.
\item[(2)]
We have a commutative diagram
\[
\xymatrix{
0 \ar[r] & \Cl(K'') \ar[r] & B^S(L)_{G''} \ar[r] & \bigoplus_{v \in S} \Z[\dcoset{G''}{G}{I_v}] \ar[r] & \Z \ar[r] & 0\\
0 \ar[r] & \Cl(K') \ar[r] \ar[u]& B^S(L)_{G'} \ar[r] \ar[u] & \bigoplus_{v \in S} \Z[\dcoset{G'}{G}{I_v}] \ar[r] \ar@{^(->}[u] & \Z \ar[r] \ar@{^(->}[u] & 0,
}
\]
where the left vertical arrow is the capitulation map and the other vertical arrows are induced by a norm element of $G'/G''$; in particular, the right-most map is the multiplication by $[G': G'']$.
\end{itemize}
\end{lem}

\subsection{Consequences}\label{ss:ram_app}

\subsubsection{The norm map}

An analogue of Proposition \ref{prop:norm_unr} is obtained if $S$ is a singleton:

\begin{prop}\label{prop:norm_ram}
For a finite Galois extension $L/K$ that is unramified outside a singleton $S = \{v\}$ of a place of $K$, setting $G = \Gal(L/K)$, we have an exact sequence
\[
H_1(G, \Delta(G/I_v)) \to \Cl(L)_G \to \Cl(K) \to \Delta(G/I_v)_G \to 0.
\]
Here, $\Delta(G/I_v)$ denotes the kernel of the map $\Z[G/I_v] \to \Z$.
\end{prop}

\begin{proof}
This follows from \eqref{eq:fund_ram}, \eqref{eq:fund_ram2}, and Lemma \ref{lem:funct_ram}(1).
\end{proof}

\begin{eg}\label{eg:sin_ram_cyc}
Suppose moreover that $G$ is a cyclic group.
Then we have $H_2(G, \Z) = 0$ and the map $H_1(G, \Z[G/I_v]) \to H_1(G, \Z)$ is identified with the inclusion map $I_v \hookrightarrow G$.
We thus conclude
\[
0 \to \Cl(L)_G \to \Cl(K) \to G/I_v \to 0.
\]
We also obtain
\[
\# \Cl(L)^G 
= \# \Cl(L)_G
=\# \Cl(K) \cdot \frac{\# I_v}{\# G}.
\]
This is again a very special case of Chevalley's ambiguous class number formula.
\end{eg}

\subsubsection{The capitulation map}

The following corresponds to Proposition \ref{prop:unr_cap}, although a complete description of the cokernel is missing.

\begin{prop}\label{prop:ram_cap}
For a finite Galois extension $L/K$ that is unramified outside a finite nonempty set $S$ of places of $K$, 
we have an exact sequence
\[
0 \to \hat{H}^{-1}(G, B^S(L)) \to \Cl(K) \to \Cl(L)^G \to \hat{H}^0(G, B^S(L)).
\]
\end{prop}

\begin{proof}
This is proved by applying the snake lemma to the commutative diagram
\[
\xymatrix{
0 \ar[r] & \Cl(L)^G \ar[r] & B^S(L)^G \ar[r] & \bigoplus_{v \in S} \Z[G/I_v]^G\\
0 \ar[r] & \Cl(K) \ar[r] \ar[u] & B^S(L)_G \ar[u] \ar[r] & \bigoplus_{v \in S} \Z \ar@{^(->}[u]
}
\]
obtained by Lemma \ref{lem:funct_ram}(2).
\end{proof}

It is unclear whether this proposition is useful for studying capitulation kernels in the spirit of Suzuki's theorem in Remark \ref{rem:Suzuki}.

\subsection{Analogues for infinite extensions}\label{ss:ram_inf}

The same argument as in Subsection \ref{ss:ram_pf} applies to infinite extensions $L/K$ as well.
In contrast to the unramified case, there even exist abelian infinite extensions that are unramified outside a finite set $S$ of places of $K$; for instance, any $\Z_p^d$-extension has this property.

\begin{defn}\label{defn:Y_S(L)}
Let $K$ be a number field and $S$ a nonempty finite set of places of $K$.
Let $L/K$ be a (possibly infinite) Galois extension unramified outside $S$ and set $G = \Gal(L/K)$.
We take an auxiliary (possibly infinite) extension $F/L$ satisfying the following conditions:
\begin{itemize}
\item[(a)]
$F\supset H(L)$,
\item[(b)]
$F/L$ is unramified outside $S$, and
\item[(c)]
$F/K$ is a Galois extension.
\end{itemize}
For example, the minimal choice is $F = H(L)$.
Alternatively, the maximal choice is $F = K_S$, the maximal $S$-ramified (not necessarily abelian) extension of $K$.
We then set $\cG = \Gal(F/K)$ (by condition (c)).

For each $v \in S$, we take a place $u_v$ of $F$ lying above $v$, and write $w_v$ for the place of $L$ lying below $u_v$.
We then write $\cI_v \subset \cG$ and $I_v \subset G$ for the inertia subgroup of $u_v$ and $w_v$, respectively.
Note that the natural map $\cG \twoheadrightarrow G$ induces $\cI_v \twoheadrightarrow I_v$.

We now define a left $\hatZ[[\cG]]$-module $\cY^{S, F}(L)$ as the kernel of the augmentation map
\[
\bigoplus_{v \in S} \hatZ[[\cG/\cI_v]] \to \hatZ,
\]
which is surjective since $S \neq \emptyset$.
We finally define
\[
Y^S(L) = (\cY^{S, F}(L))_{\Gal(F/L)}
= \hatZ[[G]] \otimes_{\hatZ[[\cG]]} \cY^{S, F}(L),
\]
which is a left $\hatZ[[G]]$-module.
This $Y^S(L)$ is independent of the choice of $F$, up to isomorphism (this follows in the same way as Remark \ref{rem:B_unr_indep}).
\end{defn}

This module $Y^S(L)$ satisfies the following.

\begin{thm}\label{thm:main_ram_inf}
For any intermediate field $K'$ of $L/K$ with $G' = \Gal(L/K')$, we have an exact sequence
\[
0 \to X(K') \to Y^S(L)_{G'} \to \bigoplus_{v \in S} \hatZ[[\dcoset{G'}{G}{I_v}]] \to \hatZ \to 0.
\]
\end{thm}

\begin{proof}
When $K'/K$ is a finite extension, the set $\dcoset{G'}{G}{\cI_v}$ is finite for each $v \in S$ and the same proof as in Theorem \ref{thm:main_ram} is valid.
Then the general case follows by taking the projective limit.
\end{proof}

By using this module $Y^S(L)$, we deduce two known propositions concerning the relation between $X(L)$ and $X(K')$ for the intermediate fields $K'$ of $L/K$.
The first one requires a restrictive assumption.

\begin{prop}\label{prop:ICNF_11}
Suppose that $S = \{v\}$ is a singleton and $v$ is totally ramified in $L/K$.
Then $X(L)_{G'} \simeq X(K')$ for any intermediate field $K'$ of $L/K$ with $G' = \Gal(L/K')$.
\end{prop}

\begin{proof}
Under the assumption, we have $X(L) \simeq Y^S(L)$ and $X(K') \simeq Y^S(L)_{G'}$ by the property in Theorem \ref{thm:main_ram_inf}.
Then the proposition follows immediately.
Alternatively, this can be deduced by taking the projective limit of Proposition \ref{prop:norm_ram}.
\end{proof}

The second one is useful in Iwasawa theory (see Proposition \ref{prop:ICNF_1}).

\begin{prop}\label{prop:ICNF_12}
Let $K'$ be an intermediate field of $L/K$ with $G' = \Gal(L/K')$.
Suppose that $L/K'$ is totally ramified at all places lying above a place in $S$.
Then the map $X(L) \to X(K')$ is surjective and we have
\[
\Ker(X(L) \to X(K')) \simeq \Ker(Y^S(L) \to Y^S(L)_{G'}) = (\widehat{\Delta} G')Y^S(L),
\]
where $\widehat{\Delta} G'$ denotes the augmentation ideal of $\hatZ[[G']]$.
\end{prop}

\begin{proof}
This follows from the diagram (see Lemma \ref{lem:funct_ram}(1))
\[
\xymatrix{
0 \ar[r] & X(L) \ar[r] \ar@{->>}[d] & Y^S(L) \ar[r] \ar@{->>}[d] & \bigoplus_{v \in S} \hatZ[[G/I_v]] \ar[r] \ar[d]^{\simeq} &
\hatZ \ar[r] \ar@{=}[d] &
0 \\
0 \ar[r] &
X(K') \ar[r] &
Y^S(L)_{G'} \ar[r] & \bigoplus_{v \in S} \hatZ[[\dcoset{G'}{G}{I_v}]] \ar[r] & \hatZ \ar[r] &
0,
}
\]
where the isomorphism follows from the assumption on the ramification.
\end{proof}

We now deduce Theorem \ref{thm:main_ram_Iw}.

\begin{proof}[Proof of Theorem \ref{thm:main_ram_Iw}]
We define $Y_{K_{\infty}} = \Z_p \otimes_{\hatZ} Y^S(K_{\infty})$.
Then \eqref{eq:Iw_inf} and \eqref{eq:Iw_fin} are valid.
To check that $Y_{K_{\infty}}$ is finitely generated as a $\Z_p[[\Gamma]]$-module, we only have to see that $(Y_{K_{\infty}})_{\Gamma}$ is finitely generated over $\Z_p$ by \eqref{eq:Iw_fin} and then apply the topological version of Nakayama's lemma.

To check that $Y_{K_{\infty}}$ is torsion requires more argument.
Recall that $I_v$ is non-trivial for any $v \in S$ by the choice of $S$.
Therefore, by \eqref{eq:Iw_inf}, the claim is reduced to the well-known fact that $X_{K_{\infty}}$ is torsion.
This fact is a part of Theorem \ref{thm:intro_duality} and independently reproved in Subsection \ref{ss:pf_final}.
\end{proof}

\section{Variants}\label{sec:variants}

In this short section, we discuss analogues of the results in Sections \ref{sec:unr} and \ref{sec:ram} for the variants of class groups.
We formulate the theorems only for infinite extensions with coefficients in $\hatZ$.
Those can be stated for finite extensions with coefficients in $\Z$ as well, as long as the variants of class groups concerned are assumed to be finite.

Let $K$ be a number field.
Let $\Sigma$ and $\Sigma'$ be finite sets of places of $K$ such that $\Sigma \cap \Sigma' = \emptyset$.
For each algebraic extension $L/K$, we consider $X_{\Sigma}^{\Sigma'}(L) = \Gal(H_{\Sigma}^{\Sigma'}(L)/L)$, where $H_{\Sigma}^{\Sigma'}(L)/L$ denotes the maximal abelian extension of $L$ that is $\Sigma$-ramified (i.e., unramified outside $\Sigma$) and $\Sigma'$-split (i.e., the primes above $\Sigma'$ split completely).
When $\Sigma$ or $\Sigma'$ is the empty set, we omit it from the notation: We have $X_{\Sigma}(L) = X_{\Sigma}^{\emptyset}(L)$, $X^{\Sigma'}(L) = X_{\emptyset}^{\Sigma'}(L)$, and $X(L) = X_{\emptyset}^{\emptyset}(L)$ is compatible with the earlier definition.

The following generalizes Theorem \ref{thm:main_unr_inf} (where $\Sigma = \Sigma' = \emptyset$).

\begin{thm}\label{thm:ray_cl}
Let $L/K$ be a $\Sigma$-ramified and $\Sigma'$-split Galois extension and set $G = \Gal(L/K)$.
Then there is a profinite $\hatZ[[G]]$-module $Y_{\Sigma}^{\Sigma'}(L)$ such that, for any intermediate field $K'$ of $L/K$ with $G' = \Gal(L/K')$, we have an exact sequence
\[
0 \to X_{\Sigma}^{\Sigma'}(K') \to Y_{\Sigma}^{\Sigma'}(L)_{G'} \to \hatZ[[G' \backslash G]] \to \hatZ \to 0.
\]
\end{thm}

\begin{proof}
The same proof as Theorem \ref{thm:main_unr_inf} applies; we only have to replace ``unramified'' by ``$\Sigma$-ramified and $\Sigma'$-split.''
The auxiliary field $F$ can be $F = H_{\Sigma}^{\Sigma'}(L)$.
\end{proof}

For example, the resulting sequence
\[
0 \to \Z_p \otimes_{\hatZ} X_{\Sigma}(L) \to \Z_p \otimes_{\hatZ} Y_{\Sigma}(L) \to \Z_p[[G]] \to \Z_p \to 0
\]
(with $\Sigma' = \emptyset$) is nothing but the Tate sequence, which plays an important role in equivariant Iwasawa theory (see Ritter--Weiss \cite{RW02} etc.).

The following generalizes Theorem \ref{thm:main_ram_inf} (where $\Sigma = \Sigma' = \emptyset$).

\begin{thm}\label{thm:ray_cl_ram}
Let $S$ be a finite set of places of $K$ such that $S \not \subset \Sigma$.
Let $L/K$ be an $(S \cup \Sigma)$-ramified and $(\Sigma' \setminus S)$-split Galois extension and set $G = \Gal(L/K)$.
Then there is a profinite $\hatZ[[G]]$-module $Y_{\Sigma}^{\Sigma', S}(L)$ such that, for any intermediate field $K'$ of $L/K$ with $G' = \Gal(L/K')$, we have an exact sequence
\[
0 \to X_{\Sigma}^{\Sigma'}(K') \to Y_{\Sigma}^{\Sigma', S}(L)_{G'} 
\to \bigoplus_{v \in S \setminus (\Sigma \cup \Sigma')} \hatZ[[G' \backslash G / I_v]] \oplus \bigoplus_{v \in S \cap \Sigma'} \hatZ[[G' \backslash G / D_v]]
\to \hatZ \to 0,
\]
where $D_v$ denotes the decomposition group.
\end{thm}

\begin{proof}
The same proof as Theorem \ref{thm:main_ram_inf} applies.
We take an auxiliary field $F$, which can be $F = H_{\Sigma}^{\Sigma'}(L)$, and set $\mathcal{G} = \Gal(F/K)$.
For each $v \in S$, choose a place of $F$ above $v$, and let
$\mathcal{I}_v$ and $\mathcal{D}_v$ denote its inertia and decomposition
subgroups in $\mathcal{G}$, respectively.
Define $\cY_{\Sigma}^{\Sigma', S, F}(L)$ as the kernel of the map
\[
\bigoplus_{v \in S \setminus (\Sigma \cup \Sigma')} \hatZ[[\mathcal{G}/\mathcal{I}_v]] \oplus \bigoplus_{v \in S \cap \Sigma'} \hatZ[[\mathcal{G}/\mathcal{D}_v]]\to \hatZ,
\]
which is surjective since $S \not\subset \Sigma$.
We then set $Y_{\Sigma}^{\Sigma', S}(L) = \cY_{\Sigma}^{\Sigma', S, F}(L)_{\Gal(F/L)}$.
\end{proof}

\begin{proof}[Proof of Theorem \ref{thm:main_ram_Iw_sp}]
Exactly in the same way as Theorem \ref{thm:main_ram_Iw} follows from Theorem \ref{thm:main_ram_inf},
the theorem follows from Theorem \ref{thm:ray_cl_ram} applied to the case where $\Sigma = \emptyset$ and $\Sigma'$ is the set of the $p$-adic primes.
\end{proof}

\section{Iwasawa's asymptotic class number formula}\label{sec:app_Iw}

In this section, after algebraic preliminaries in Subsections \ref{ss:pre_Iw} and \ref{ss:ICNF_pf}, we deduce (a refined version of) Theorem \ref{thm:intro_ICNF} from Theorem \ref{thm:main_ram_Iw} in Subsection \ref{ss:ICNF_pf2}.

\subsection{Preliminaries on the Iwasawa algebra}\label{ss:pre_Iw}

We begin with an algebraic consideration.
In this and the next subsections, let $\Gamma$ be a profinite group isomorphic to $\Z_p$ and set $\Lambda = \Z_p[[\Gamma]]$.
Taking a topological generator $\gamma \in \Gamma$, we have an isomorphism
\[
\Lambda \simeq \Z_p[[T]],
\quad
\gamma \leftrightarrow 1 + T,
\]
which is called Serre's isomorphism.
This implies that $\Lambda$ is a complete regular local ring of dimension $2$.
It also follows that $\Lambda$ is a UFD, so any prime ideal of height one is principal.
The prime ideals of $\Lambda$ are classified according to their heights:
\begin{itemize}
\item
There is a unique prime ideal of height zero, namely, $(0)$.
\item
The prime ideals of height one are $(p)$ and $(f)$, where $f \in \Lambda \setminus p\Lambda$ is a prime element.
\item
There is a unique prime ideal of height two, namely, the maximal ideal that corresponds to $(p, T) \subset \Z_p[[T]]$.
\end{itemize}

\begin{defn}
Let $M$ be a finitely generated torsion $\Lambda$-module.
We define its characteristic ideal by
\[
\cha(M) = \prod_{\fp} \fp^{\length_{\Lambda_{\fp}}(M_{\fp})},
\]
where $\fp$ runs over the height one prime ideals of $\Lambda$.
Here, $\Lambda_{\fp}$ and $M_{\fp}$ denote the localizations.
\end{defn}

\begin{defn}
We define the $\lambda$- and $\mu$-invariants for elements of $\Lambda$ and modules over $\Lambda$ as follows.
\begin{itemize}
\item
Let $f \in \Lambda$ be a nonzero element.
We define $\mu(f) \geq 0$ as the largest integer such that $f \in p^{\mu(f)} \Lambda$.
We then define $\lambda(f) \geq 0$ as the largest integer such that $\ol{p^{-\mu(f)} f} \in T^{\lambda(f)} \F_p[[T]]$, where $\ol{(-)}$ denotes the reduction modulo $p$.
\item
Let $M$ be a finitely generated torsion $\Lambda$-module.
We define $\mu(M), \lambda(M) \in \N$ as the $\mu$- and $\lambda$-invariants of any generator of the characteristic ideal $\cha(M)$.
\end{itemize}
\end{defn}

It is clear that $\lambda, \mu: \Lambda \setminus\{0\} \to \N$ are both monoid homomorphisms.
For an exact sequence $0 \to M' \to M \to M'' \to 0$ of finitely generated torsion $\Lambda$-modules, we have
\[
\cha(M) = \cha(M') \cha(M''),
\]
that is, the characteristic ideals are additive with respect to exact sequences.
It then follows that $\lambda(M) = \lambda(M') + \lambda(M'')$ and $\mu(M) = \mu(M') + \mu(M'')$ as well.

\subsection{An abstract version of Iwasawa's asymptotic formula}\label{ss:ICNF_pf}

One of our techniques is the use of the following $\rho_{\alpha}(M)$, which behaves well in homological arguments.

\begin{defn}
Let $\alpha \in \Lambda$ be a nonzero element.
For each finitely generated $\Lambda$-module $M$, we write $M[\alpha]$ and $M/\alpha$ for the kernel and cokernel of the multiplication by $\alpha$ map on $M$.
If $M[\alpha]$ and $M/\alpha$ are both finite, we set
\[
\rho_{\alpha}(M) = \ord_p(\#(M/\alpha)) - \ord_p(\#(M[\alpha])) \in \Z.
\]
\end{defn}

Concerning this definition, it is convenient to have the following in mind.

\begin{lem}\label{lem:rho_wd}
Let $\alpha \in \Lambda$ be a nonzero element and $M$ a finitely generated $\Lambda$-module.
If $M/\alpha$ is finite, then $M[\alpha]$ is also finite.
\end{lem}

\begin{proof}
A finitely generated $\Lambda$-module $N$ is finite if and only if $N_{\fp} = 0$ for any height-one prime ideal $\fp$.
Therefore, the assumption implies that the multiplication by $\alpha$ map on $M_{\fp}$ is surjective.
It then follows that the multiplication by $\alpha$ map on $M_{\fp}$ is moreover an isomorphism, and we obtain the conclusion.
\end{proof}

We record some basic properties of $\rho_{\alpha}(M)$.

\begin{lem}\label{lem:rho_add}
Let $\alpha \in \Lambda$ be a nonzero element and $M$ a finitely generated $\Lambda$-module.
The following hold.
\begin{itemize}
\item[(1)]
Let $\alpha = \alpha' \alpha''$ in $\Lambda$.
Then $\rho_{\alpha}(M)$ is defined if and only if both $\rho_{\alpha'}(M)$ and $\rho_{\alpha''}(M)$ are defined.
In this case, we have $\rho_{\alpha}(M) = \rho_{\alpha'}(M) + \rho_{\alpha''}(M)$.
\item[(2)]
Let $0 \to M' \to M \to M'' \to 0$ be an exact sequence of $\Lambda$-modules.
Then $\rho_{\alpha}(M)$ is defined if and only if both $\rho_{\alpha}(M')$ and $\rho_{\alpha}(M'')$ are defined.
In this case, we have $\rho_{\alpha}(M) = \rho_{\alpha}(M') + \rho_{\alpha}(M'')$.
\item[(3)]
If $M$ is finite, then $\rho_{\alpha}(M) = 0$.
\end{itemize}
\end{lem}

\begin{proof}
(1)
This follows from the kernel-cokernel exact sequence
\[
0 \to M[\alpha'] \to M[\alpha] \overset{\times \alpha'}{\to} M[\alpha''] \to M/\alpha' \overset{\times \alpha''}{\to} M/\alpha \to M/\alpha'' \to 0.
\]
We also use Lemma \ref{lem:rho_wd} to see that $\rho_{\alpha''}(M)$ is defined if $\rho_{\alpha}(M)$ is defined.

(2)
This follows from the snake lemma:
\[
0 \to M'[\alpha] \to M[\alpha] \to M''[\alpha] \to M'/\alpha \to M/\alpha \to M''/\alpha \to 0.
\]
We also use Lemma \ref{lem:rho_wd} to see that $\rho_{\alpha}(M'')$ is defined if $\rho_{\alpha}(M)$ is defined.

(3)
This follows from the exact sequence
\[
0 \to M[\alpha] \to M \overset{\alpha \times}{\to} M \to M/\alpha \to 0.
\]
\end{proof}

For each $n \geq 0$, we set
\[
\omega_n = \gamma^{p^n}-1 \in \Z[\Gamma].
\]
For $n \geq 1$, we then set
\[
\varphi_n = \frac{\omega_n}{\omega_{n-1}}
= 1 + \gamma^{p^{n-1}} + \gamma^{2p^{n-1}} + \dots + \gamma^{(p-1)p^{n-1}} \in \Z[\Gamma].
\]

Another important idea in this paper is to use prime filtrations to derive the following proposition.

\begin{prop}\label{prop:ICNF_4}
Let $M$ be a finitely generated torsion $\Lambda$-module and $n \geq 1$.
Suppose that
\[
p^{n-1}(p-1) > \lambda(f)
\]
holds for any prime divisor $f$ of a generator of the characteristic ideal of $M$.
Then we have
\[
\rho_{\varphi_n}(M)
= \lambda(M) + \mu(M) (p^n - p^{n-1}).
\]
\end{prop}

\begin{proof}
We first reduce the proposition to the special case where $M = \Lambda/\fp$ for a prime ideal $\fp$.
Since $\Lambda$ is a noetherian commutative ring and $M$ is finitely generated, there exists a prime filtration
\[
0 = M_0 \subset M_1 \subset \cdots \subset M_r = M
\]
such that $M_j/M_{j-1}$ is isomorphic to $\Lambda/\fp_j$ for some prime ideal $\fp_j$.
By the additivity of $\rho_{\varphi_n}(-)$ (Lemma \ref{lem:rho_add}(2)), we have
\[
\rho_{\varphi_n}(M) = \sum_{j=1}^r \rho_{\varphi_n}(M_j/M_{j-1}),
\]
as long as each term on the right hand side is well-defined.
On the other hand, the additivity of $\lambda(-)$ and $\mu(-)$ shows
\[
\lambda(M) = \sum_{j=1}^r \lambda(M_j/M_{j-1}),
\quad
\mu(M) = \sum_{j=1}^r \mu(M_j/M_{j-1}).
\]
Therefore, it is enough to establish the claim for $M_j/M_{j-1}$ for all $j$.

In what follows, we assume $M = \Lambda/\fp$ with $\fp$ a prime ideal of $\Lambda$.
Since $M$ is torsion, we have $\fp \neq (0)$.
If $\fp$ is the maximal ideal of $\Lambda$, then $M$ is finite and both sides of the claim are zero (by Lemma \ref{lem:rho_add}(3)).
Therefore, it remains to consider the case where the height of $\fp$ is one.

Suppose $\fp = (p)$.
Then $M = \Lambda/(p) \simeq \F_p[[T]]$ and $\ol{\varphi_n}$ corresponds to $T^{p^n - p^{n-1}}$.
Therefore, we obtain
\[
\rho_{\varphi_n}(M)
= \ord_p(\#(\F_p[[T]]/(T^{p^n - p^{n-1}})))
= p^n - p^{n-1}
\]
and the claim holds.

Suppose $\fp = (f)$ for some prime element $f \in \Lambda \setminus p \Lambda$.
Set $\lambda = \lambda(f)$.
Using the identification $\Lambda \simeq \Z_p[[T]]$, we write
\[
f = a_0 + a_1 T + \cdots + a_{\lambda} T^{\lambda} + \cdots,
\]
so we have $p \nmid a_{\lambda}$ and $p \mid a_i$ for $i = 0, 1, \dots, \lambda - 1$.
We take a primitive $p^n$-th root of unity $\zeta_{p^n}$.
Then we have an isomorphism
\[
\Lambda/(\varphi_n) \simeq \Z_p[\zeta_{p^n}]
\]
by sending $\gamma$ to $\zeta_{p^n}$.
It follows that
\[
M/\varphi_n
\simeq (\Lambda/(\varphi_n))/f
\simeq
\Z_p[\zeta_{p^n}]/(f(\zeta_{p^n}-1)).
\]
Let $\ord_{\Z_p[\zeta_{p^n}]}(-)$ be the valuation of $\Q_p(\zeta_{p^n})$ whose image is $\Z \cup \{\infty\}$.
Since
\[
\ord_{\Z_p[\zeta_{p^n}]}(p) 
= p^{n-1}(p-1)
> \lambda
= \ord_{\Z_p[\zeta_{p^n}]}((\zeta_{p^n}-1)^{\lambda})
\]
by the assumption, we see that 
\[
\lambda = \ord_{\Z_p[\zeta_{p^n}]}(a_{\lambda} (\zeta_{p^n}-1)^{\lambda}) 
< \ord_{\Z_p[\zeta_{p^n}]}(a_i (\zeta_{p^n}-1)^i)
\]
for any $i \neq \lambda$.
Therefore,
\[
\ord_{\Z_p[\zeta_{p^n}]}(f(\zeta_{p^n}-1))
= \lambda,
\]
which shows $\ord_p(\# (M/\varphi_n)) = \lambda$.
In particular, $M/\varphi_n$ is finite, so $f$ and $\varphi_n$ are coprime to each other and we have $M[\varphi_n] = 0$.
Consequently, we obtain
\[
\rho_{\varphi_n}(M)
= \ord_p(\# (M/\varphi_n))
= \lambda.
\]
This completes the proof.
\end{proof}

\subsection{Iwasawa's asymptotic formula}\label{ss:ICNF_pf2}

Let $K_{\infty}/K$ be a $\Z_p$-extension with $\Gamma = \Gal(K_{\infty}/K)$.
Let $Y_{K_{\infty}}$ be the finitely generated torsion $\Z_p[[\Gamma]]$-module constructed in Theorem \ref{thm:main_ram_Iw}.
As noted in its proof, we accept the fact that $X_{K_{\infty}}$ is torsion.
For $\Z_p$-extensions, this fact can be alternatively deduced by using Proposition \ref{prop:ICNF_1} below, though we omit the proof.

For each place $v$ of $K$ that is ramified in $K_{\infty}/K$, the inertia group $I_v$ is open in $\Gamma$ since $\Gamma \simeq \Z_p$ and $I_v$ is non-trivial.
It follows that there is an integer $m \geq 0$ such that every ramified place is totally ramified in $K_{\infty}/K_m$.

Now Proposition \ref{prop:ICNF_12} is restated in this case as follows.
We fix a topological generator $\gamma \in \Gamma$ and define $\omega_n$ and $\varphi_n$ as in the paragraph before Proposition \ref{prop:ICNF_4}.

\begin{prop}\label{prop:ICNF_1}
Let $m \geq 0$ be an integer such that every ramified place is totally ramified in $K_{\infty}/K_m$.
Then, for each $n \geq m$, we have an exact sequence
\[
0 \to \omega_n Y_{K_{\infty}} \to X_{K_{\infty}} \to A_{K_n} \to 0.
\]
In particular, we have
\[
\frac{\# A_{K_n}}{\# A_{K_m}} = \# \left( \frac{\omega_m Y_{K_{\infty}}}{\omega_n Y_{K_{\infty}}} \right).
\]
\end{prop}

This proposition plays a central role in the derivation of Iwasawa's asymptotic class number formula, both in the original proof (Iwasawa \cite{Iwa59}; see also Washington \cite{Was97}) and in our approach.
The original proof of this proposition itself requires more technical arguments.
In particular, it relies heavily on the assumption that $K_{\infty}/K$ is a $\Z_p$-extension, and requires the choice of an auxiliary $p$-adic place.
In our proof presented above, those techniques are no longer necessary.

We now prove the main theorem.

\begin{thm}[Iwasawa's asymptotic class number formula]\label{thm:main_ICNF}
Let $n_0 \geq m \geq 0$ be integers satisfying the following:
\begin{itemize}
\item[(a)]
every ramified place is totally ramified in $K_{\infty}/K_m$,
\item[(b)]
$p^{n_0}(p-1) > \lambda(f)$ holds for any prime divisor $f$ of $\cha(X_{K_{\infty}})$, and
\item[(c)]
$(X_{K_{\infty}})_{\fin}$ is annihilated by $\omega_{n_0}/\omega_ m$, where $(-)_{\fin}$ denotes the maximal finite $\Lambda$-submodule.
\end{itemize}
Then there exists $\nu \in \Z$ such that
\[
\ord_p(\# \Cl(K_n))
= \lambda n + \mu p^n + \nu
\]
holds for all $n \geq n_0$ with $\lambda = \lambda(X_{K_{\infty}})$ and $\mu = \mu(X_{K_{\infty}})$.
\end{thm}

\begin{proof}
We write $Y = Y_{K_{\infty}}$.
Note first that $\omega_m Y$ is a submodule of $X_{K_{\infty}}$ of finite index by Proposition \ref{prop:ICNF_1}, so conditions (b) and (c) also hold for $\omega_m Y$ instead of $X_{K_{\infty}}$, and their $\lambda$- and $\mu$-invariants coincide as well.

For $n \geq m$, we set
\[
\nu_{m, n} 
= \frac{\omega_n}{\omega_m} = \varphi_{m+1} \cdots \varphi_n.
\]
By Proposition \ref{prop:ICNF_1} and condition (a), we have
\begin{align*}
\ord_p(\# \Cl(K_n)) - \ord_p(\# \Cl(K_m)) 
& = \ord_p(\# (\omega_m Y / \omega_n Y))\\
& = \ord_p(\# ((\omega_m Y)/\nu_{m, n}))\\
& = \rho_{\nu_{m, n}}(\omega_m Y) + \ord_p(\#((\omega_m Y)[\nu_{m, n}])).
\end{align*}
By Lemma \ref{lem:rho_add}(1), we have
\[
\rho_{\nu_{m, n}}(\omega_m Y)
= \rho_{\nu_{m, m+1}}(\omega_m Y) + \cdots + \rho_{\nu_{n-1, n}}(\omega_m Y).
\]
Then for $n \geq n_0$, Proposition \ref{prop:ICNF_4} and condition (b) show
\[
\rho_{\nu_{m, n}}(\omega_m Y)
 = \rho_{\nu_{m, n_0}}(\omega_m Y) + \lambda(\omega_m Y)(n - n_0) + \mu(\omega_m Y) (p^n - p^{n_0}).
\]
We also have $(\omega_m Y)[\nu_{m, n}] = (\omega_m Y)_{\fin}$ by condition (c).
This completes the proof.
\end{proof}

\begin{rem}
Here are remarks on conditions (b) and (c).
\begin{itemize}
\item
The usual argument relying on the structure theorem of finitely generated $\Lambda$-modules naturally requires a condition that $p^{n_0}(p-1) > \lambda(f)$ for each direct summand $\Lambda/(f)$ of the elementary module pseudo-isomorphic to $X_{K_{\infty}}$.
Our argument relying on the prime filtrations and the additivity of $\rho_{\alpha}(-)$ weakens this condition to condition (b).

\item
A sufficient condition for (c) is $n_0 - m \geq \length_{\Lambda}((X_{K_{\infty}})_{\fin}) = \ord_p(\# (X_{K_{\infty}})_{\fin})$.
This follows by considering the descending series
\[
(X_{K_{\infty}})_{\fin} 
\supset \nu_{m, m+1}  (X_{K_{\infty}})_{\fin}
\supset \nu_{m, m+2}  (X_{K_{\infty}})_{\fin}
\supset \cdots
\]
and using Nakayama's lemma; if an equality holds then the term is already zero.
\end{itemize}
\end{rem}

Let us mention an earlier result of Sands \cite[Theorem 3.1]{San91} giving the explicit range of $n$.
It is shown that the asymptotic formula holds for any $n \geq 0$ under the following assumptions:
\begin{itemize}
\item
every ramified place is totally ramified in $K_{\infty}/K$,
\item
$\mu(X_{K_{\infty}}) = 0$ and $\lambda(X_{K_{\infty}}) < p-1$ hold, and
\item
$(X_{K_{\infty}})_{\fin} = 0$ holds.
\end{itemize}
This is recovered from Theorem \ref{thm:main_ICNF} by taking $n_0 = m = 0$ (the assumption $\mu(X_{K_{\infty}}) = 0$ can be removed).

\section{Duality of class groups along $\Z_p^d$-extensions}\label{sec:duality}

In this section, we deduce Theorem \ref{thm:intro_duality} and Corollary \ref{cor:GGC_equiv} as consequences of Theorem \ref{thm:main_ram_Iw}.
As already discussed in Subsection \ref{ss:intro_duality}, the proof largely consists of a general study of the functors $\fG$ and $\fF$.
In Subsection \ref{ss:gen_funct}, we study linear functors in general.
By using those results, we study the functor $E^1 = \Ext^1_{\Lambda}(-, \Lambda)$ in Subsection \ref{ss:E1} and the functors $\fG$ and $\fF$ in Subsection \ref{ss:FG}.
Unlike in Section \ref{sec:app_Iw}, we use the structure theorem for $\Lambda$-modules.
The main theorems will be proved in Subsection \ref{ss:pf_final}.

\subsection{General observations on linear functors}\label{ss:gen_funct}

Throughout Subsections \ref{ss:gen_funct}--\ref{ss:FG}, let $\Gamma$ be a profinite group that is isomorphic to $\Z_p^d$ for some $d \geq 1$.
We set $\Lambda = \Z_p[[\Gamma]]$, which is known to be isomorphic to $\Z_p[[T_1, \dots, T_d]]$ (Serre's isomorphism).

Recall that a $\Lambda$-module $M$ is said to be pseudo-null if $M$ is finitely generated and $M_\fp = 0$ holds for any prime ideal $\fp$ of $\Lambda$ with $\height(\fp) \leq 1$.
A $\Lambda$-homomorphism $\alpha: M \to N$ between finitely generated modules is said to be a pseudo-isomorphism if both $\Ker(\alpha)$ and $\Coker(\alpha)$ are pseudo-null.
This is equivalent to the assertion that the induced homomorphism $\alpha_\fp: M_\fp \to N_\fp$ is an isomorphism for any prime ideal $\fp$ of $\Lambda$ with $\height(\fp) \leq 1$.
Two finitely generated torsion modules $M$ and $N$ are said to be pseudo-isomorphic if there exists a pseudo-isomorphism $\alpha: M \to N$.
It is known that this is an equivalence relation.

Let $\modC{\Lambda}$ denote the category of finitely generated $\Lambda$-modules.
We only state results for covariant functors for simplicity, and the obvious variants for contravariant functors are also valid.
A covariant functor $\cF: \modC{\Lambda} \to \modC{\Lambda}$ is said to be $\Lambda$-linear if, for any $M$ and $N$, the map
\[
\Hom_{\Lambda}(M, N) \to \Hom_{\Lambda}(\cF(M), \cF(N)),
\quad
\alpha \mapsto \cF(\alpha)
\]
is a $\Lambda$-homomorphism.

For such a $\Lambda$-linear functor $\cF$, it is easy to see that the annihilator of $M$ is contained in the annihilator of $\cF(M)$.
Indeed, if $s \in \Lambda$ annihilates $M$, then $s \id_M = 0$ implies
\[
s \id_{\cF(M)} = s \cF(\id_M) = \cF(s \id_M) = \cF(0) = 0.
\]
In particular, $\cF(M)$ is torsion (resp.~pseudo-null) if so is $M$.
We will show in Proposition \ref{prop:funct_2}(1), moreover, that $\cF(\alpha)$ is a pseudo-isomorphism if so is $\alpha$.

\begin{lem}\label{lem:loc_hom}
Let $\fp$ be a prime ideal of $\Lambda$.
For finitely generated $\Lambda$-modules $M$ and $N$, we have a natural isomorphism
\[
\Hom_{\Lambda}(N, M)_{\fp} \simeq \Hom_{\Lambda_{\fp}}(N_{\fp}, M_{\fp})
\]
induced by $\alpha \mapsto \alpha_{\fp}$ for $\alpha \in \Hom_{\Lambda}(N, M)$.
\end{lem}

\begin{proof}
When $N$ is a free module, the claim is clear.
For general $N$, we can take a finite presentation $F' \to F \to N \to 0$ with finitely generated free modules $F$ and $F'$.
Then the claim follows from the claim for $F$ and $F'$, together with the exactness of the localization and the left exactness of $\Hom$; we omit the details.
\end{proof}

\begin{lem}\label{lem:funct_3}
Let $\fp$ be a prime ideal of $\Lambda$.
Let $\cF: \modC{\Lambda} \to \modC{\Lambda}$ be a $\Lambda$-linear functor.
If $\alpha: M \to N$ is a homomorphism in $\modC{\Lambda}$ such that $\alpha_\fp: M_{\fp} \to N_{\fp}$ is an isomorphism, then $\cF(\alpha)_\fp: \cF(M)_{\fp} \to \cF(N)_{\fp}$ is also an isomorphism.
\end{lem}

\begin{proof}
We first note that, for any homomorphism $\alpha: M \to N$ in $\modC{\Lambda}$, the following are equivalent:
\begin{itemize}
\item[(i)]
The induced homomorphism $\alpha_{\fp}: M_{\fp} \to N_{\fp}$ is an isomorphism.
\item[(ii)]
There exist a homomorphism $\beta: N \to M$ and an element $s \in \Lambda \setminus \fp$ such that $\alpha \circ \beta = s \id_N$ and $\beta \circ \alpha = s \id_M$ hold.
\end{itemize}
This follows from Lemma \ref{lem:loc_hom}.

By (i) $\Rightarrow$ (ii), we can take $\beta: N \to M$ and $s \in \Lambda \setminus \fp$ such that $\alpha \circ \beta = s \id_N$ and $\beta \circ \alpha = s \id_M$ hold.
Then by applying the functor $\cF$, we obtain
\[
\cF(\alpha) \circ \cF(\beta) = s \id_{\cF(N)},
\quad
\cF(\beta) \circ \cF(\alpha) = s \id_{\cF(M)}.
\]
Therefore, by (ii) $\Rightarrow$ (i), we see that $\cF(\alpha)_\fp$ is an isomorphism.
\end{proof}

For a finitely generated $\Lambda$-module $M$, let $M_{\tors}$ be its torsion submodule.

\begin{prop}\label{prop:funct_2}
Let $\cF: \modC{\Lambda} \to \modC{\Lambda}$ be a $\Lambda$-linear functor.
The following hold.
\begin{itemize}
\item[(1)]
For a pseudo-isomorphism $\alpha: M \to N$, the homomorphism $\cF(\alpha): \cF(M) \to \cF(N)$ is a pseudo-isomorphism as well.
\item[(2)]
Suppose that $\cF(\Lambda)$ is torsion.
Then for any finitely generated $\Lambda$-module $M$, the module $\cF(M)$ is torsion.
\item[(3)]
Suppose that $\cF(\Lambda)$ is pseudo-null.
Then for any finitely generated $\Lambda$-module $M$, the homomorphism $\cF(M_{\tors}) \to \cF(M)$ induced by the inclusion map $M_{\tors} \hookrightarrow M$ is a pseudo-isomorphism.
\end{itemize}
\end{prop}

\begin{proof}
(1)
We only have to apply Lemma \ref{lem:funct_3} to all $\fp$ with $\height(\fp) \leq 1$.

(2)(3)
Let $r$ be the rank of $M$ as a $\Lambda$-module.
Let $\fp$ be any prime with $\height(\fp) \leq 1$; we need $\height(\fp) = 0$ (i.e.,~$\fp = (0)$) for (2).
Since $\Lambda_{\fp}$ is either a field or a discrete valuation ring, the module $(M/M_{\tors})_{\fp} \simeq M_{\fp} / (M_{\tors})_{\fp}$ is isomorphic to $\Lambda_{\fp}^r$.
Therefore, by Lemma \ref{lem:loc_hom}, we can take a homomorphism $\Lambda^r \to M$ such that the induced homomorphism
\[
\Lambda^r \oplus M_{\tors} \to M
\]
becomes an isomorphism after base-change to $\Lambda_{\fp}$.
Then by Lemma \ref{lem:funct_3}, the induced homomorphism
\[
\cF(\Lambda)^r \oplus \cF(M_{\tors}) \to \cF(M)
\]
also becomes an isomorphism after base-change to $\Lambda_{\fp}$.
This shows claims (2) and (3).
\end{proof}

We will study three functors $E^1$, $\fG$, and $\fF$, which are all $\Lambda$-linear.
Moreover, all of these send the module $\Lambda$ to $0$, so this proposition is applicable.
When $d = 1$, Lee \cite{Lee20} showed that the particular functors $\fG$ and $\fF$ preserve pseudo-isomorphisms.

\subsection{The functor $E^1$}\label{ss:E1}

For a finitely generated $\Lambda$-module $M$, we define 
\[
E^1(M) = \Ext^1_{\Lambda}(M, \Lambda),
\]
which is again finitely generated since $\Lambda$ is noetherian.
Here, for a $\Lambda$-module $M$, we equip $\Hom_{\Lambda}(M, \Lambda)$ with an action of $\Lambda$ by
\[
(\sigma f)(x) = f(\sigma x)
\]
for $f \in \Hom_{\Lambda}(M, \Lambda)$, $\sigma \in \Gamma$, and $x \in M$.

We regard $E^1$ as a contravariant $\Lambda$-linear functor $E^1: \modC{\Lambda} \to \modC{\Lambda}$.
This functor already plays important roles in Iwasawa theory (for instance, see Jannsen \cite{Jan89}).
We include proofs of known results to demonstrate how Proposition \ref{prop:funct_2} works.

\begin{lem}\label{lem:E1_basic}
The following hold.
\begin{itemize}
\item[(1)]
We have $E^1(\Lambda) = 0$.
\item[(2)]
For any $f \in \Lambda \setminus \{0\}$, we have an isomorphism $E^1(\Lambda/f) \simeq \Lambda/f$.
\end{itemize}
\end{lem}

\begin{proof}
(1)
This is because $\Lambda$ is projective as a $\Lambda$-module.

(2)
We have a tautological short exact sequence
\[
0 \to \Lambda \overset{\times f}{\to} \Lambda \to \Lambda/f \to 0.
\]
By using $\Hom_{\Lambda}(\Lambda, \Lambda) \simeq \Lambda$, this induces an exact sequence 
\[
0 \to \Lambda \overset{f \times}{\to} \Lambda \to E^1(\Lambda/f) \to 0,
\]
so we obtain the claim.
\end{proof}

\begin{prop}\label{prop:E1_pisom}
Let $M$ be a finitely generated $\Lambda$-module.
\begin{itemize}
\item[(1)]
The module $E^1(M)$ is finitely generated and torsion.
\item[(2)]
The modules $E^1(M)$ and $M_{\tors}$ are pseudo-isomorphic to each other.
\end{itemize}
\end{prop}

\begin{proof}
(1)
This follows from Proposition \ref{prop:funct_2}(2) and $E^1(\Lambda) = 0$.

(2)
By Proposition \ref{prop:funct_2}(3) and $E^1(\Lambda) = 0$, the homomorphism $E^1(M) \to E^1(M_{\tors})$ is a pseudo-isomorphism.
By the structure theorem for finitely generated torsion $\Lambda$-modules, there exists a pseudo-isomorphism
\[
\alpha: M_{\tors} \to N = \bigoplus_{i=1}^s \Lambda/f_i
\]
with $f_1, \dots, f_s \in \Lambda \setminus \{0\}$.
By Proposition \ref{prop:funct_2}(1), the homomorphism $E^1(\alpha): E^1(N) \to E^1(M_{\tors})$ is a pseudo-isomorphism.
By Lemma \ref{lem:E1_basic}(2), there exists a (non-canonical) isomorphism $E^1(N) \simeq N$.
Therefore,
\[
M_{\tors} \overset{\alpha}{\to} N \simeq E^1(N) \overset{E^1(\alpha)}{\to} E^1(M_{\tors})
\]
is a pseudo-isomorphism.
This completes the proof.
\end{proof}

Note that this proposition can also be proved by using the long exact sequences of $\Ext^{\bullet}_{\Lambda}(-, \Lambda)$ as well as the structure theorem for finitely generated $\Lambda$-modules.
However, the functors $\fG$ and $\fF$ do not behave well with respect to exact sequences (see Remark \ref{rem:FG_not_ex}).
The argument given above also applies to $\fG$ and $\fF$.

\subsection{The functors $\fG$ and $\fF$}\label{ss:FG}

As in Subsection \ref{ss:intro_duality}, we define two functors $\fG$ and $\fF$.

\begin{defn}
Let $M$ be a finitely generated $\Lambda$-module.
\begin{itemize}
\item[(1)]
We define
\[
\fG(M) = \varprojlim_U (M_U)_{\fin},
\]
where $U$ runs over the open subgroups of $\Gamma$ and the transition map is induced by the projection map $M_V \to M_U$ for $U \supset V$.
\item[(2)]
We define
\[
\fF(M) = \varprojlim_U ((M_U)_{\fin})^{\vee},
\quad
\]
where the transition map is induced by the map $M_U \to M_V$ induced by the norm element of $U/V$ for $U \supset V$.
The action of $\Gamma$ on $\fF(M)$ is defined as in Subsection \ref{ss:intro_duality}.
\end{itemize}
\end{defn}

Then $\fG$ is a covariant functor, while $\fF$ is a contravariant functor, and both are $\Lambda$-linear.
The following lemma ensures that $\fG(M)$ and $\fF(M)$ are finitely generated.

\begin{lem}\label{lem:G_fg}
Let $M$ be a finitely generated $\Z_p[[\Gamma]]$-module.
\begin{itemize}
\item[(1)]
We have a natural injective homomorphism $\fG(M) \hookrightarrow M$.
Therefore, $\fG(M)$ is finitely generated.
\item[(2)]
We have a natural injective homomorphism $\fF(M) \hookrightarrow E^1(M)$.
Therefore, $\fF(M)$ is finitely generated.
\end{itemize}
\end{lem}

\begin{proof}
(1) This follows immediately from 
\[
\fG(M) = \varprojlim_U (M_U)_{\fin} 
\subset \varprojlim_U M_U
\simeq M.
\]

(2)
This follows directly from the exact sequence 
\[
0 \to \varprojlim_U ((M_U)_{\fin})^{\vee} \to E^1(M) \to \varprojlim_U (H_1(U, M) \otimes (\Q_p/\Z_p))^{\vee} \to 0
\]
proved by Jannsen \cite[Theorem 2.1]{Jan89} (reproduced in Neukirch--Schmidt--Wingberg \cite[Theorem (5.4.13)]{NSW08}).
Indeed, we have natural isomorphisms
\[
((M_U)_{\fin})^{\vee} 
\simeq \Ext^1_{\Z_p}(M_U, \Z_p)
\simeq \Ext^1_{\Z_p[\Gamma/U]}(M_U, \Z_p[\Gamma/U])
\]
and the sequence above is a consequence of the universal coefficient theorem.
\end{proof}

The exact sequence in the proof of (2) above is also used by Lim \cite[Theorem 2.3]{Lim24} in a similar way.
In that work, $M$ is assumed to be ``systematically coinvariant-finite,'' which implies that $\fF(M) \simeq E^1(M)$.
The condition is satisfied by the higher $K$-groups, but in this paper we have to deal with modules that are not systematically coinvariant-finite.

The following lemma plays the role of Lemma \ref{lem:E1_basic}.

\begin{lem}\label{lem:FG_bas}
\begin{itemize}
\item[(1)]
We have $\fG(\Lambda) = 0$ and $\fF(\Lambda) = 0$.
\item[(2)]
For any $f \in \Lambda \setminus \{0\}$, we have an isomorphism $\fG(\Lambda/f) \simeq \fF(\Lambda/f)$.
\end{itemize}
\end{lem}

\begin{proof}
(1)
These follow from $(\Lambda_U)_{\fin} = \Z_p[\Gamma/U]_{\fin} = 0$ for any $U$.

(2)
For a while, we fix an open subgroup $U$ of $\Gamma$ and set $G = \Gamma/U$.
Then $(\Lambda/f)_U \simeq \Z_p[G]/f$, where we simply write $f$ for the image of $f$.
We will construct an isomorphism
\[
((\Z_p[G]/f)_{\fin})^{\vee} \simeq (\Z_p[G]/f)_{\fin},
\]
that is, a perfect pairing
\[
(\Z_p[G]/f)_{\fin} \times (\Z_p[G]/f)_{\fin} \to \Q_p/\Z_p,
\]
which is moreover compatible when $U$ varies.

Let $J = J_U(f) \subset \Z_p[G]$ be the submodule such that $J/(f) = (\Z_p[G]/f)_{\fin}$.
In other words, $J$ is the $\Z_p$-saturation of the ideal $(f) \subset \Z_p[G]$.
We define $(-, -): J \times J \to \Q_p[G]$ roughly by $(x, y) = xy/f$, although this requires justification since $f$ is not a non-zero-divisor in general.
To be precise, we define
\[
(x, y) = \xi y
\]
by taking $\xi \in \Q_p[G]$ such that $x = f \xi$.
One checks that this is well-defined, i.e., independent of the choice of $\xi$.
We also have an alternative definition $(x, y) = x \eta$ by taking $\eta \in \Q_p[G]$ such that $y = f \eta$.
This is a $\Z_p[G]$-bilinear pairing, so we have $(gx, y) = (x, gy) = g(x, y)$ for $g \in G$.

This pairing induces a $\Z_p[G]$-bilinear pairing
\[
(-, -): J/(f) \times J/(f) \to \Q_p[G]/\Z_p[G].
\]
Indeed, if $x \in (f)$, then we can take $\xi \in \Z_p[G]$, so $(x, y) = \xi y \in \Z_p[G]$.

We claim that this is a non-degenerate pairing.
Let $x \in J$ be an element such that $(x, y) \in \Z_p[G]$ holds for any $y \in J$.
Take an element $\xi \in \Q_p[G]$ such that $x = f \xi$.
Then the assumption says $\xi y \in \Z_p[G]$ for any $y \in J$, so we can consider the diagram
\[
\xymatrix{
	J \ar[r]^-{\times \xi} \ar@{^(->}[d] & \Z_p[G] \\
	\Z_p[G] \ar@{..>}[ru]_{\times \xi'}
}
\]
We have
\[
\Ext^1_{\Z_p[G]}(\Z_p[G]/J, \Z_p[G])
\simeq \Ext^1_{\Z_p}(\Z_p[G]/J, \Z_p)
= 0
\]
since $J$ is saturated in $\Z_p[G]$.
Therefore, there exists an element $\xi' \in \Z_p[G]$ making the diagram commute.
In particular, for the element $f \in J$, we obtain $f \xi = f \xi'$.
This shows that $x \in (f)$, as claimed.

By composing with the map $\Q_p[G]/\Z_p[G] \to \Q_p/\Z_p$ that takes the coefficients of the identity element, we obtain a $\Z_p$-bilinear pairing
\[
\lrang{-,-}: (\Z_p[G]/f)_{\fin} \times (\Z_p[G]/f)_{\fin} \to \Q_p/\Z_p
\]
such that
\[
(x, y) 
= \sum_{g \in G} \lrang{g^{-1} x, y} g
= \sum_{g \in G} \lrang{x, g^{-1}y} g.
\]
One checks that $\lrang{-, -}$ is also a perfect pairing.

We now vary the open subgroup $U$, so we write $(-, -)_U$ and $\lrang{-, -}_U$ instead of $(-, -)$ and $\lrang{-, -}$.
For each $U \supset V$, the following commutes:
\[
\xymatrix{
\lrang{-,-}_V:  \hspace{-1em}
& (\Z_p[\Gamma/V]/f)_{\fin} \hspace{-1em} \ar[d]_{\pi}
& \hspace{-1em} \times \hspace{-1em}
& (\Z_p[\Gamma/V]/f)_{\fin}
  \ar[r]^-{(-,-)_V}
& \Q_p[\Gamma/V]/\Z_p[\Gamma/V] \ar[r]
& \Q_p/\Z_p \ar@{=}[d]
\\
\lrang{-,-}_U:  \hspace{-1em}
& (\Z_p[\Gamma/U]/f)_{\fin} \hspace{-1em}
& \hspace{-1em} \times \hspace{-1em}
& (\Z_p[\Gamma/U]/f)_{\fin}
  \ar[r]_-{(-,-)_U} \ar[u]_{\nu}
& \Q_p[\Gamma/U]/\Z_p[\Gamma/U]
  \ar[r] \ar[u]_{\nu}
& \Q_p/\Z_p,
}
\]
where $\pi$ denotes the projection and $\nu$ denotes the norm map.
Indeed, for $\ol{x} \in (\Z_p[\Gamma/V]/f)_{\fin}$ and $\ol{y} \in (\Z_p[\Gamma/U]/f)_{\fin}$, taking $\xi \in \Q_p[\Gamma/V]$ such that $x = f \xi$, we have $\pi(x) = f \pi(\xi)$, so
\[
\nu (\pi(x), y)_U = \nu (\pi(\xi)y) = \xi \nu(y) = (x, \nu(y))_V,
\]
as claimed.

Therefore, we obtain
\[
\fF(\Lambda/f)
= \varprojlim_U (((\Lambda/f)_U)_{\fin})^{\vee}
\simeq \varprojlim_U ((\Z_p[\Gamma/U]/f)_{\fin})^{\vee}
\simeq \varprojlim_U (\Z_p[\Gamma/U]/f)_{\fin}
= \fG(\Lambda/f).
\]
This completes the proof.
\end{proof}

We are ready to prove the main algebraic theorem.

\begin{thm}\label{thm:FG_pisom}
Let $M$ be a finitely generated $\Lambda$-module.
\begin{itemize}
\item[(1)]
Both $\fG(M)$ and $\fF(M)$ are finitely generated torsion.
\item[(2)]
$\fG(M)$ and $\fF(M)$ are pseudo-isomorphic to each other.
\end{itemize}
\end{thm}

As already remarked, this is shown by Lee when $d = 1$.

\begin{proof}
This is shown in the same manner as Proposition \ref{prop:E1_pisom}.

(1)
They are finitely generated by Lemma \ref{lem:G_fg}.
Then we apply Proposition \ref{prop:funct_2}(2) and Lemma \ref{lem:FG_bas}(1) to see that they are torsion.

(2)
By Proposition \ref{prop:funct_2}(3) and Lemma \ref{lem:FG_bas}(1), the homomorphisms $\fG(M_{\tors}) \to \fG(M)$ and $\fF(M) \to \fF(M_{\tors})$ are both pseudo-isomorphisms.
By the structure theorem, there exists a pseudo-isomorphism
\[
\alpha: M_{\tors} \to N = \bigoplus_{i=1}^s \Lambda/f_i
\]
with $f_1, \dots, f_s \in \Lambda \setminus \{0\}$.
By Proposition \ref{prop:funct_2}(1), the homomorphisms $\fG(\alpha): \fG(M_{\tors}) \to \fG(N)$ and $\fF(\alpha): \fF(N) \to \fF(M_{\tors})$ are pseudo-isomorphisms.
By Lemma \ref{lem:FG_bas}(2), there exists a (non-canonical) isomorphism $\fG(N) \simeq \fF(N)$.
Therefore,
\[
\fG(M_{\tors}) \overset{\fG(\alpha)}{\to} \fG(N) \simeq \fF(N) \overset{\fF(\alpha)}{\to} \fF(M_{\tors})
\]
is a pseudo-isomorphism.
This completes the proof.
\end{proof}

It appears to be useful to record a picture (not a commutative diagram):
\[
\xymatrix{
	\fG(M) \ar@{^(->}[r] \ar@{}[d]|{\rotatebox{90}{$\sim$}} & M_{\tors} \ar@{}[d]|{\rotatebox{90}{$\sim$}} \\
	\fF(M) \ar@{^(->}[r] & E^1(M).
}
\]
Note, however, that the vertical pseudo-isomorphisms are not canonical.

We conclude this general theory by observing some examples and remarks.

\begin{eg}\label{eg:F_comp}
Let $f \in \Lambda \setminus \{0\}$.
As in the proof of Lemma \ref{lem:FG_bas}(2), for each open subgroup $U$ of $\Gamma$, let $J_U(f) \supset (f)$ be the ideal of $\Z_p[\Gamma/U]$ such that
\[
(\Z_p[\Gamma/U]/f)_{\fin} = J_U(f)/(f).
\]
We define an ideal of $\Lambda$ by $J(f) = \varprojlim_U J_U(f) \supset (f)$.
Then we obtain $\fG(\Lambda/f)= J(f)/(f)$.
We now consider two extreme cases:
\begin{itemize}
\item
Suppose the image of $f$ in $\Z_p[\Gamma/U]$ is a non-zero-divisor for any $U$.
Then we have $(\Z_p[\Gamma/U]/f)_{\fin} = \Z_p[\Gamma/U]/f$, so we have $J(f) = \Lambda$ and $\fG(\Lambda/f) = \Lambda/f$.
\item
Suppose $f$ is of the form $\gamma - 1$ for an element $\gamma \in \Gamma \setminus \{1\}$.
Then we have $(\Z_p[\Gamma/U]/f)_{\fin} = 0$ for any $U$.
In this case, we have $J(f) = (f)$ and $\fG(\Lambda/f) = 0$.
\end{itemize}
Another description of $J_U(f)$ is $J_U(f) = \Z_p[\Gamma/U] \cap f \Q_p[\Gamma/U]$.
Since $\Q_p[\Gamma/U]$ is a product of fields, we have $f \Q_p[\Gamma/U] = f^e \Q_p[\Gamma/U]$ and so
\[
J(f^e) = J(f)
\]
holds for any $e \geq 1$.
It follows that $\fG(\Lambda/f^e) = J(f)/(f^e)$.
For example, if $f$ is of the form $\gamma - 1$ for an element $\gamma \in \Gamma \setminus \{1\}$, we have $\fG(\Lambda/f^e) = (f)/(f^e) \subset \Lambda/(f^e)$.
\end{eg}

\begin{rem}\label{rem:FG_not_ex}
Let us observe that $\fG$ is not an exact functor even if we ignore pseudo-null modules.
Let $f = \gamma - 1$ with $\gamma \in \Gamma \setminus \{1\}$.
Then
\[
0 \to \Lambda/f \overset{\times f}{\to} \Lambda/f^2 \to \Lambda/f \to 0
\]
is an exact sequence, but applying $\fG$ yields
\[
0 \to 0 \to (f)/(f^2) \to 0 \to 0
\]
by Example \ref{eg:F_comp}, which is not exact even up to pseudo-null modules.
This example, together with Lemma \ref{lem:FG_bas}(2), also shows that $\fF$ is not exact.
\end{rem}

\begin{rem}\label{rem:FG_pisom_2}
For any finitely generated $\Z_p[[\Gamma]]$-module $M$, by Lemma \ref{lem:G_fg}, we have a natural homomorphism
\[
\fF(M) \to E^1(M) \to E^1(\fG(M)).
\]
Although the modules $\fF(M)$ and $E^1(\fG(M))$ are known to be pseudo-isomorphic to each other, this homomorphism need not be a pseudo-isomorphism.
For example, let $f = \gamma - 1$ with $\gamma \in \Gamma \setminus \{1\}$ and consider $M = \Lambda/(f^2)$.
Then this composite map is identified with
\[
(f)/(f^2) \hookrightarrow \Lambda/(f^2) \overset{\times f}{\to} (f)/(f^2),
\]
which is the zero map and so is not a pseudo-isomorphism.
\end{rem}

\subsection{The proofs of the main results}\label{ss:pf_final}

Let $K_{\infty}/K$ be a $\Z_p^d$-extension and set $\Gamma =\Gal(K_{\infty}/K)$ and $\Lambda = \Z_p[[\Gamma]]$.
We use the finitely generated $\Lambda$-module $Y_{K_{\infty}}$ constructed in Theorem \ref{thm:main_ram_Iw}.
Here we do not need to assume the fact that $X_{K_{\infty}}$ is torsion and indeed reprove it.

\begin{proof}[Proof of Theorem \ref{thm:intro_duality}]
As already discussed in Subsection \ref{ss:intro_duality}, by using Lemma \ref{lem:funct_ram}, we have
\[
X_{K_{\infty}} \simeq \fG(Y_{K_{\infty}}),
\quad
A_{K_{\infty}}^{\vee} \simeq \fF(Y_{K_{\infty}}).
\]
Since $Y_{K_{\infty}}$ is known to be finitely generated over $\Lambda$, we are able to apply Theorem \ref{thm:FG_pisom}.
The theorem follows.
\end{proof}

To deduce Corollary \ref{cor:GGC_equiv}, we use the following algebraic fact.
We say that a $\Lambda$-module $M$ is PN-free if $M$ has no nonzero pseudo-null submodules.

\begin{prop}\label{prop:PN_free}
Let $M$ be a finitely generated torsion $\Lambda$-module.
Then both $E^1(M)$ and $\fF(M)$ are PN-free.
\end{prop}

\begin{proof}
The claim for $E^1(M)$ is well-known and we omit the proof (see, for example, Vauclair \cite[Lemme 5.1(iii)]{Vau09}).
Then the claim for $\fF(M)$ also follows by the existence of the injective homomorphism $\fF(M) \hookrightarrow E^1(M)$ (Lemma \ref{lem:G_fg}(2)).
\end{proof}

Note that the assumption that $M$ is torsion cannot be dropped.
For instance, if $\Lambda = \Z_p[[T]]$ and $M = (p, T) \subset \Lambda$, then $E^1(M) \simeq \F_p$ is not PN-free.

\begin{proof}[Proof of Corollary \ref{cor:GGC_equiv}]
The equivalence (i) $\Leftrightarrow$ (ii) is clear from Theorem \ref{thm:intro_duality}.
Since we know that $Y_{K_{\infty}}$ is a finitely generated and torsion module, Proposition \ref{prop:PN_free} implies that $A_{K_{\infty}}^{\vee} \simeq \fF(Y_{K_{\infty}})$ is PN-free.
Therefore, (ii) $\Leftrightarrow$ (iii) holds.
\end{proof}

\section*{Use of AI}

The author used ChatGPT to improve the exposition.

\section*{Acknowledgments}

The author thanks Cornelius Greither, Takamichi Sano, Takuya Yanagisawa, and Taichi Komori for their helpful comments on this work.
This work is supported by JSPS KAKENHI Grant Number 26K16965.

{
\bibliographystyle{abbrv}
\bibliography{biblio}

\begin{thebibliography}{10}

\bibitem{CM81}
A.~A. Cuoco and P.~Monsky.
\newblock Class numbers in {${\bf Z}\sp{d}\sb{p}$}-extensions.
\newblock {\em Math. Ann.}, 255(2):235--258, 1981.

\bibitem{GS64}
E.~S. Golod and I.~R. \v{S}afarevi\v{c}.
\newblock On the class field tower.
\newblock {\em Izv. Akad. Nauk SSSR Ser. Mat.}, 28:261--272, 1964.

\bibitem{Gree01}
R.~Greenberg.
\newblock Iwasawa theory---past and present.
\newblock In {\em Class field theory---its centenary and prospect ({T}okyo,
  1998)}, volume~30 of {\em Adv. Stud. Pure Math.}, pages 335--385. Math. Soc.
  Japan, Tokyo, 2001.

\bibitem{Gru76}
K.~W. Gruenberg.
\newblock {\em Relation modules of finite groups}.
\newblock American Mathematical Society, Providence, R.I., 1976.
\newblock Conference Board of the Mathematical Sciences Regional Conference
  Series in Mathematics, No. 25.

\bibitem{GW00}
K.~W. Gruenberg and A.~Weiss.
\newblock Capitulation and transfer kernels.
\newblock {\em J. Th\'eor. Nombres Bordeaux}, 12(1):219--226, 2000.

\bibitem{Iwa59}
K.~Iwasawa.
\newblock On {$\Gamma $}-extensions of algebraic number fields.
\newblock {\em Bull. Amer. Math. Soc.}, 65:183--226, 1959.

\bibitem{Iwa73}
K.~Iwasawa.
\newblock On {${\bf Z}\sb{l}$}-extensions of algebraic number fields.
\newblock {\em Ann. of Math. (2)}, 98:246--326, 1973.

\bibitem{Jan89}
U.~Jannsen.
\newblock Iwasawa modules up to isomorphism.
\newblock In {\em Algebraic number theory}, volume~17 of {\em Adv. Stud. Pure
  Math.}, pages 171--207. Academic Press, Boston, MA, 1989.

\bibitem{LT21}
K.~F. Lai and K.-S. Tan.
\newblock A generalized {I}wasawa's theorem and its application.
\newblock {\em Res. Math. Sci.}, 8(2):Paper No. 20, 18, 2021.

\bibitem{Lee20}
J.~Lee.
\newblock Structure of the {M}ordell-{W}eil group over the
  {$\Bbb{Z}_p$}-extensions.
\newblock {\em Trans. Amer. Math. Soc.}, 373(4):2399--2425, 2020.

\bibitem{Lei17}
A.~Lei.
\newblock Estimating class numbers over metabelian extensions.
\newblock {\em Acta Arith.}, 180(4):347--364, 2017.

\bibitem{LL19}
D.~Liang and M.~F. Lim.
\newblock On the {I}wasawa asymptotic class number formula for
  {$\Bbb{Z}_p^r\rtimes \Bbb {Z}_p$}-extensions.
\newblock {\em Acta Arith.}, 189(2):191--208, 2019.

\bibitem{Lim24}
M.~F. Lim.
\newblock Comparing direct limit and inverse limit of even {$K$}-groups in
  multiple {$\Bbb Z_p$}-extensions.
\newblock {\em J. Th\'{e}or. Nombres Bordeaux}, 36(2):537--555, 2024.

\bibitem{Nek06}
J.~Nekov\'{a}\v{r}.
\newblock Selmer complexes.
\newblock {\em Ast\'{e}risque}, 310:viii+559, 2006.

\bibitem{NSW08}
J.~Neukirch, A.~Schmidt, and K.~Wingberg.
\newblock {\em Cohomology of number fields}, volume 323 of {\em Grundlehren der
  Mathematischen Wissenschaften [Fundamental Principles of Mathematical
  Sciences]}.
\newblock Springer-Verlag, Berlin, second edition, 2008.

\bibitem{RW96}
J.~Ritter and A.~Weiss.
\newblock A {T}ate sequence for global units.
\newblock {\em Compositio Math.}, 102(2):147--178, 1996.

\bibitem{RW02}
J.~Ritter and A.~Weiss.
\newblock Toward equivariant {I}wasawa theory.
\newblock {\em Manuscripta Math.}, 109(2):131--146, 2002.

\bibitem{San91}
J.~W. Sands.
\newblock On small {I}wasawa invariants and imaginary quadratic fields.
\newblock {\em Proc. Amer. Math. Soc.}, 112(3):671--684, 1991.

\bibitem{Suz91}
H.~Suzuki.
\newblock A generalization of {H}ilbert's theorem 94.
\newblock {\em Nagoya Math. J.}, 121:161--169, 1991.

\bibitem{Vau09}
D.~Vauclair.
\newblock Sur la dualit\'{e} et la descente d'{I}wasawa.
\newblock {\em Ann. Inst. Fourier (Grenoble)}, 59(2):691--767, 2009.

\bibitem{Was97}
L.~C. Washington.
\newblock {\em Introduction to cyclotomic fields}, volume~83 of {\em Graduate
  Texts in Mathematics}.
\newblock Springer-Verlag, New York, second edition, 1997.

\end{thebibliography}
}

\end{document}